\documentclass[a4paper,11pt]{amsart}
\usepackage[bb,wider,cthms]{gewoehn} 
\usepackage{color,graphicx}
\usepackage[bookmarks=false,pdfpagemode=UseNone]{hyperref}
\definecolor{linkc}{rgb}{0,0.2,0.6}
\hypersetup{pdfstartview=XYZ null null 1.27, pdfpagemode=UseOutlines, 
	citecolor = {black},colorlinks,breaklinks,urlcolor=linkc,linkcolor={black}}
\newcommand{\D}{\mathbb{D}}
\renewcommand{\L}{\mathbb{L}}

\begin{document}

\title[Spaces of curves with constrained curvature on hyperbolic surfaces]
{Spaces of curves with constrained curvature \\
on hyperbolic surfaces} 
\author{Nicolau C.~Saldanha} 
\author{Pedro Z\"{u}hlke}
\maketitle

\begin{abstract} 
	Let $ S $ be a hyperbolic surface. We investigate the topology of the space
	of all curves on $ S $ which start and end at given points in given
	directions, and whose curvatures are constrained to lie in a given interval
	$ (\ka_1,\ka_2) $. Such a space falls into one of four qualitatively
	distinct classes, according to whether $ (\ka_1,\ka_2) $ contains, overlaps,
	is disjoint from, or contained in the interval $ [-1,1] $. Its homotopy type
	is computed in the latter two cases.  We also study the behavior of these
	spaces under covering maps when $ S $ is arbitrary (not necessarily
	hyperbolic nor orientable) and show that if $ S $ is compact, then they are
	always nonempty.
\end{abstract}


\setcounter{section}{-1}
\section{Introduction}\label{S:intro}

A curve on a surface $ S $ is said to be \tdfn{regular} if it has a continuous
and nonvanishing derivative.  Consider two such curves whose curvatures take
values in some given interval, starting in the same direction (prescribed by a
unit vector tangent to $ S $) and ending in another prescribed direction. It is
a natural problem to determine whether one curve can be deformed into the other
while keeping end-directions fixed and respecting the curvature bounds.  From
another viewpoint, one is asking for a characterization of the connected
components of the space of all such curves.  More ambitiously, what is its
homotopy or homeomorphism type? The answer can be unexpectedly interesting, and
it is closely linked to the geometry of $ S $. See the discussion below on some
related results.

In this article this question is investigated when $ S $ is hyperbolic, that is,
a (possibly nonorientable) smooth surface endowed with a complete Riemannian
metric of constant negative curvature, say, $ -1$. It is well known that
any such surface can be expressed as the quotient of the hyperbolic plane $
\Hh^2 $ by a discrete group of isometries. Thus any curve on $ S $ can be lifted
to a curve in $ \Hh^2 $ having the same curvature (at least in absolute value,
if $ S $ is nonorientable). As will be more carefully explained later, this
implies that one can obtain a solution of the problem about spaces of curves on
$ S $ if one knows how to solve the corresponding problem in the hyperbolic
plane for all pairs of directions.  

Let $ u,\,v \in  UT\Hh^2 $ (the unit tangent bundle of $ \Hh^2 $) be two such
directions.  Let $ \sr C_{\ka_1 }^{\ka_2 }(u,v) $ be the set of all smooth
regular curves on $ \Hh^2 $ whose initial (resp.~terminal) unit tangent vectors
equal $ u $ (resp.~$ v $) and whose curvatures take values inside $
(\ka_1,\ka_2) $, furnished with the $ C^\infty $-topology. There are canonical
candidates for the position of connected component of this space, defined as
follows. Let
\begin{equation*}
	\pr\colon \wt{UT\Hh^2} \to UT\Hh^2 
\end{equation*}
denote the universal cover. Fixing a lift $ \te{u} $ of $ u $, one obtains a map
$ \sr C_{\ka_1}^{\ka_2}(u,v) \to \pr^{-1}(v) $ by looking at the endpoint of the
lift to the universal cover, starting at $ \te{u} $, of curves in $ \sr
C_{\ka_1 }^{\ka_2 }(u,v) $. Since $ \pr^{-1}(v) $ is discrete, this yields a
decomposition of $ \sr C_{\ka_1 }^{\ka_2 }(u,v)  $ into closed-open subspaces.

More concretely, let $ \ga \colon [0,1] \to \C $ be a regular plane curve. An
\tdfn{argument} $ \theta\colon [0,1] \to \R $ for $ \ga $ is a continuous
function such that $ \dot\ga $ always points in the direction of $ e^{i\theta}
$; note that there are countably many such functions, which differ by multiples
of $ 2\pi $. The \tdfn{total turning} of $ \ga $ is defined as $ \theta(1) -
\theta(0)$.  Because $ \Hh^2 $ is diffeomorphic to (an open subset of) $ \C $,
any regular curve in the former also admits arguments and a total turning.
However, these have no geometric meaning since they depend on the choice of
diffeomorphism (and subset). In any case, once such a choice has been made, it
gives rise to the decomposition \begin{equation}\label{E:decomp} \sr
	C_{\ka_1}^{\ka_2}(u,v) = \Du_{\tau} \sr C_{\ka_1}^{\ka_2}(u,v;\tau),
\end{equation} where $ \sr C_{\ka_1}^{\ka_2}(u,v; \tau) $ consists of those
curves in $ \sr C_{\ka_1 }^{\ka_2 }(u,v) $ which have total turning $ \tau $,
and $ \tau $ runs over all \tdfn{valid} total turnings, viz., those for which $
v $ is parallel to $ e^{i\tau}u $ (regarded as vectors in $ \C $).  The
closed-open subspaces appearing on the right side of \eqref{E:decomp}, which
will be referred to as the \tdfn{canonical subspaces} of $ \sr
C_{\ka_1}^{\ka_2}(u,v) $, are independent of the diffeomorphism: they are in
bijective correspondence with the fiber $ \pr^{-1}(v) $. If $ \ka_1=-\infty $
and $ \ka_2=+\infty $, so that no restriction is imposed on the curvature, then
they are in fact precisely the components of $ \sr C_{\ka_1 }^{\ka_2 }(u,v) $,
and each of them is contractible.  However, for general curvature bounds they
may fail to be connected, contractible, or even nonempty. The possibilities
depend above all upon the relation of $ (\ka_1,\ka_2) $ to the interval $ [-1,1]
$.

The main results of the paper are \tref{T:disjoint}, \tref{T:contained} and
\tref{T:compact}. The former two together assert the following.
\begin{uthm}\label{T:main}
	Let $ u,\,v \in UT\Hh^2 $ and $ \ka_1 < \ka_2 $.
	\begin{enumerate}
		\item [(a)] If $ (\ka_1,\ka_2) \subs [-1,1] $, then at
			most one of the canonical subspaces of $ \sr C_{\ka_1}^{\ka_2}(u,v)
			$ is nonempty. This subspace is always contractible. Thus $ \sr
			C_{\ka_1 }^{\ka_2 }(u,v) $ itself is either empty or contractible.
		\item [(b)]  If $ (\ka_1,\ka_2) $ is disjoint from $ [-1,1] $, then
			infinitely many of the canonical subspaces are empty, and infinitely
			many are nonempty.  The latter are all contractible.
	\end{enumerate}
\end{uthm}

If $ (\ka_1,\ka_2) $ contains $ [-1,1] $, then none of the canonical subspaces
is empty. We conjecture that most are contractible, but finitely
many of them may have the homotopy type of an $ n $-sphere, for some $ n \in \N
$ depending upon the subspace. This conjecture is partly motivated by our
knowledge of the homotopy type of the corresponding spaces of curves in the
Euclidean plane, as determined in \cite{SalZueh2} and \cite{SalZueh1}. 

Finally, if $ (\ka_1,\ka_2) $ overlaps $ [-1,1] $, then infinitely many of
the canonical subspaces are empty, and infinitely many are nonempty.
Nevertheless, as in the previous case, simple examples (not discussed in this
article) show that these subspaces may not be connected.  We hope to determine
the homotopy type in these two last cases in a future paper. 

To state the third main result \tref{T:compact} referred to above, let $ \sr
CS_{\ka_1}^{\ka_2}(u,v) $ denote the space of curves on $ S $ starting
(resp.~ending) in the direction of $ u $ (resp.~$ v $) $ \in UTS $ whose curvatures
take values in $ (\ka_1,\ka_2) $.

\begin{uthm}
	Let $S$ be any compact connected surface (not necessarily hyperbolic nor
	orientable). Then $ \sr CS_{\kappa_1}^{\kappa_2}(u,v)\neq \emptyset $ for
	any choice of $\ka_1<\ka_2$ and $u,\,v\in UTS$.
\end{uthm}

\subsection*{Related results}A map $ f \colon M^m \to N^n $ is said to
be $ k $-th order \tit{free} if the $ k $-th order osculating space of $ f $ at
any $ p \in M $, which is generated by all (covariant) partial derivatives of $
f $ at $ p $ of order $ \leq k $, has the maximum dimension, $ d(k,m):= {m+k
\choose k} - 1 $, for all $ p \in M $. Notice that a map is first-order free if
and only if it is an immersion.

Although this definition of ``freeness'' uses covariant derivatives for
simplicity, it is easy to avoid Riemannian metrics by using the language
of jets. On the other hand, this formulation suggests the study of spaces of
maps $ f\colon M \to N $ satisfying more general differential inequalities, or
partial differential relations. Various versions of Gromov's h-principle (see
\cite{Gromov} and \cite{EliMis}) are known which permit one to reduce the
questions of existence, density and approximation of holonomic (i.e., true)
solutions to a partial differential relation to the corresponding questions
about formal (i.e., virtual) solutions; the latter are usually settled by
invoking simple facts from homotopy theory.\footnote{The discussion here is
not meant to present an exhaustive compilation of the related literature. We
apologize to any authors whose work has not been cited.}

It is known that if $ n \geq d(k,m) + 1 $, or if $ n = d(k,m) $
but $ M $ is an open manifold, then the h-principle holds for $ k $-th order
free maps $ M^m \to N^n $ (see \cite{Gromov}, p.~9). In contrast, if $ n $ equals the
critical dimension $ d(k,m) $ but $ M $ is not open, then practically nothing is
known regarding such maps.  As an example, it is an open problem to decide whether
a second-order free map $ \Ss^1 \times \Ss^1 \to \R^5 $ exists,
cf.~\cite{EliMis}. Indeed, it can be quite hard to disprove the validity of the
h-principle even for the simplest differential relations. 

In \cite{SalZueh, SalZueh2, SalZueh1} and the current article we study spaces 
of curves whose curvature is constrained to a given interval $ (\ka_1,\ka_2) $
on an elliptic, flat or hyperbolic surface. Such curves are holonomic solutions 
of a second-order differential relation on maps from a one-dimensional
manifold (an interval, or a circle) to a two-dimensional manifold (the surface),
so that we are in the critical dimension.  In all of these cases we obtain some
results on the homeomorphism type of the space of such curves, and in particular
show that they do not abide to the h-principle. Other articles on the same
topic, for curves in the Euclidean plane, include \cite{Ayala, AyaRub, Dubins1}.

If $ (\ka_1,\ka_2) = (-\infty,+\infty) $, then no condition is imposed on the
curve except that it should be an immersion (i.e., regular). This problem was
solved by H.~Whitney \cite{Whitney} for closed curves in the plane and by
S.~Smale for closed curves on any manifold \cite{Smale}. The case of immersions
of higher-dimensional manifolds has also been elucidated, among others by S.~Smale,
R.~Lashof and M.~Hirsch in \cite{Hirsch, Hirsch1, Smale3, Smale1, Smale2}.
For results concerning the topology of spaces of immersions into space forms
of nonpositive curvature having constrained principal curvatures (second
fundamental form), see \cite{Zuehlke2}.

If $ (\ka_1,\ka_2) = (0,+\infty) $, then we are asking that the curvature of the
curve never vanish. This is equivalent to the requirement that it be
second-order free. Such curves are also called \tit{nondegenerate} or
\tit{locally convex} in the literature. Works on this problem include
\cite{KheSha, KheSha1, Little, Saldanha3}; see also \cite{Arnold}. More
generally, $ n $-th order free curves on $ n $-dimensional manifolds have been
studied by several authors over the years; we mention \cite{AlvSal, Anisov,
Feldman, Feldman1, Fenchel, Goulart, Little1, MosSad, SalSha, ShaSha, Wintgen}.

Finally, in another direction, there is a very extensive literature on
applications of paths of constrained curvature in engineering and control
theory. Perhaps the main problem in this regard is the determination of
length-minimizing or other sort of optimal paths within this and related
classes. We refer the reader to \cite{Ayala2, Dubins, Mittenhuber, Mittenhuber1,
Monroy, ReeShe} for further information and references.


\subsection*{Outline of the sections}

After briefly introducing some notation and definitions, \S \ref{S:basic}
begins with a discussion of curves of constant curvature $ \ka $ in the
hyperbolic plane: circles ($ \abs{\ka} > 1 $), horocycles ($ \abs{\ka} = 1 $)
and hypercircles ($ \abs{\ka} < 1 $). Then a transformation is defined which
takes a curve and shifts it by a fixed distance along the direction prescribed
by its normal unit vectors. Its effect on the regularity and 
curvature of the original curve is investigated, and applied to reduce the
dependence of the topology of $ \sr C_{\ka_1 }^{\ka_2 }(u,v) $ to four real
parameters, instead of the eight needed to specify $ \ka_1,\,\ka_2,\,u $ and $ v
$.

In \S\ref{S:voidness} the voidness of the canonical subspaces is discussed. It
is proven that if $ (\ka_1,\ka_2) $ is contained in $ [-1,1] $, then the image
of any curve in $ \sr C_{\ka_1 }^{\ka_2 }(u,v) $ is the graph of a function when
seen in an appropriate conformal model of the hyperbolic plane, which we call
the Mercator model.\footnote{It is very likely that the
	Mercator model has already appeared under other (possibly standard) names
in the literature;  however, we do not know any references. Hypercircles are also
known as ``hypercycles'' or ``equidistant curves''.} In particular, at
most one of its canonical subspaces is nonempty. If $ (\ka_1,\ka_2) $ contains $
[-1,1] $, then none of the  canonical subspaces is empty. In the two remaining
cases, there is a critical value $ \tau_0 $ for the total turning such that $
\sr C_{\ka_1 }^{\ka_2 }(u,v;\tau) $ is nonempty for all $ \tau \geq \tau_0 $ and
empty for all $ \tau < \tau_0 $ (or reversely, depending on whether $
(\ka_1,\ka_2) $ contains points to the right or to the left of $ [-1,1] $). 

Section \ref{S:frame} explains how a curve of constrained curvature may also be
regarded as a curve in the group of orientation-preserving isometries of $ \Hh^2
$ satisfying certain conditions. This perspective is sometimes useful.

In \S\ref{S:disjoint} it is shown that the nonempty canonical subspaces 
of $ \sr C_{\ka_1 }^{\ka_2 }(u,v) $ are all contractible when $ (\ka_1,\ka_2) $
is disjoint from $ [-1,1] $. The idea is to parametrize all curves in such a
subspace by the argument of its unit tangent vector when viewed as a curve in
the half-plane model, and to take (Euclidean) convex combinations. The proof
intertwines various Euclidean and hyperbolic concepts, and seems to be highly
dependent on this particular model.

In the Mercator model $ M $, any curve in $ \sr C_{\ka_1 }^{\ka_2 }(u,v) $ may
be parametrized as $ x \mapsto (x,y(x)) $ when $ (\ka_1,\ka_2) $ is
contained in $ [-1,1] $, so the bounds on its curvature can be translated into
two differential inequalities involving $ \dot y $ and $ \ddot y $, but not $ y
$ itself, because vertical translations are isometries of $ M $. This allows one
to produce a contraction of the space by working with the associated family of
functions $ \dot y $, as carried out in \S\ref{S:contained}.

In \S\ref{S:general} we define spaces of curves with constrained
curvature on a general surface $ S $, not necessarily hyperbolic, complete nor
orientable, and explain how a  Riemannian covering of $ S $ induces
homeomorphisms between spaces of curves on $ S $ and on the covering space. We
also show that if $ S $ is compact then any such space is nonempty; this is also
true in the Euclidean plane, but not in the hyperbolic plane, as mentioned
previously.  The section ends with a brief discussion of spaces of closed curves
without basepoint conditions. 

Several useful constructions, notably one which is essential to the proof of
the main result of \S\ref{S:contained}, may create discontinuities of the
curvature, and thus lead out of the class of spaces $ \sr CS_{\ka_1 }^{\ka_2
}(u,v)  $ defined in \S\ref{S:general}. To circumvent this, the paraphernalia of
$ L^2 $ functions is used in \S\ref{S:discontinuous} to define another family of
spaces, denoted $ \sr LS_{\ka_1}^{\ka_2}(u,v) $, which are Hilbert manifolds.
The curves in these spaces are regular but their curvatures are only defined
almost everywhere. The main result of \S\ref{S:discontinuous} states that the
natural inclusion $ \sr CS_{\ka_1 }^{\ka_2 }(u,v) \inc \sr LS_{\ka_1 }^{\ka_2
}(u,v)  $ is a homotopy equivalence with dense image. Sections \ref{S:general}
and \ref{S:discontinuous} are independent of the other ones.

A few exercises are included in the article. These are
never used in the main text, and their solutions consist either of 
straightforward computations or routine extensions of arguments presented
elsewhere.  

It is assumed that the reader is familiar with the geometry of the
hyperbolic plane as discussed, for instance, in chapter 2 of \cite{Thurston},
the expository article \cite{Cannon}, chapter 7 of \cite{Beardon} or chapters
3--5 of \cite{Ratcliffe}. 

\section{Basic definitions and results}\label{S:basic}

When speaking of the hyperbolic plane with no particular model in mind, we
denote it by $ \Hh^2 $.  The underlying sets of the (Poincar\'e) disk,
half-plane and hyperboloid models are denoted by:
\begin{alignat*}{9}
	D & =\set{z \in \C}{\abs{z}<1}; \\
	H & =\set{z \in \C}{\Im(z)>0}; \\
	L & =\set{(x_0,x_1,x_2) \in \E^{2,1}}
	{-x_0^2+x_1^2+x_2^2=-1\text{\ and\ }x_0>0}.
\end{alignat*}
The circle at infinity is denoted by $ \Ss^1_\infty $ or $ \bd \Hh^2 $, the norm
of a vector $ v $ tangent to $ \Hh^2 $ by $ \abs{v} $, and the
Riemannian metric by $ \gen{\ ,\ } $. When working with $ D $ or $ H $, both $
\Hh^2 $ and its tangent planes are regarded as subsets of $ \C $

In the disk and half-plane models, we will select the orientation which is
induced on the respective underlying sets by the standard orientation of $ \C $.
In the hyperboloid model, a basis $ (u,v) $ of a tangent plane is declared
positive if $ (u,v,e_0) $ is positively oriented in $ \R^{3} $; equivalently,
the Lorentzian vector product $ u\ten v $ points to the exterior region bounded
by $ L $ (i.e., the one containing the light-cone). 

Given a regular curve $ \ga\colon [0,1]\to \Hh^2 $, its \tdfn{unit tangent} is
the map 
\begin{equation*}
	\ta=\ta_\ga\colon[0,1]\to UT\Hh^2,\quad 
	\ta := \frac{\dot \ga}{\abs{\dot \ga}}.
\end{equation*}
Let $ J \colon T\Hh^2 \to T\Hh^2 $ denote the bundle map (``multiplication by $
i$'') which associates to $ v \neq 0 $ the unique vector $ Jv $ of the same norm
as $ v $ such that $ (v,Jv) $ is orthogonal and positively oriented. Then the
\tdfn{unit normal} to $ \ga $ is given by
\begin{equation*}
	\no=\no_\ga\colon [0,1]\to UT\Hh^2, \quad \no:= J \circ \ta.
\end{equation*}
Assuming that $ \ga $ has a second derivative, its
\tdfn{curvature} is the function  
\begin{equation}\label{E:curv}
	\ka=\ka_\ga \colon [0,1] \to \R,\quad
	\ka:=\frac{1}{\abs{\dot\ga}}\Big\langle\frac{D\ta}{dt},\no\Big\rangle =
	\frac{1}{\abs{\dot\ga}^2}\Big\langle\frac{D\dot\ga}{dt},\no\Big\rangle;
\end{equation} 
here $ D $ denotes covariant differentiation (along $ \ga $).  The hyperboloid
model is usually the most convenient one for carrying out computations, since it
realizes $ \Hh^2 $ as a submanifold of the vector space $ \E^{2,1} $.  For
instance, the curvature of a curve $ \ga $ on $ L $ is given by:
\begin{equation*}\label{E:curv2}
	\ka= \frac{\dot \ta \cdot \no}{\norm{\dot \ga}} =
	\frac{\ddot \ga \cdot \no}{\norm{\dot\ga}^2},
\end{equation*}
where $ \cdot $ denotes the bilinear form on $ \E^{2,1} $ and $ \norm{\ }^2 $
is the associated quadratic form.

\begin{dfn}[spaces of curves in $ \Hh^2 $]\label{D:spaces}
	Let $ u,\,v \in UT\Hh^2 $ and $ \ka_1 < \ka_2 \in \R \cup \se{\pm \infty} $.
	Then $ \sr C_{\ka_1 }^{\ka_2 }(u,v) $ (resp.~$ \bar{\sr C}_{\ka_1
	}^{\ka_2 }(u,v) $) denotes the set of $ C^r $ regular curves $ \ga \colon
	[0,1] \to \Hh^2 $ satisfying:
	\begin{enumerate}
		\item [(i)] $ \ta_\ga(0) = u $ and $ \ta_\ga(1) = v $;
		\item [(ii)] $ \ka_1 < \ka_\ga < \ka_2 $ (resp.~$ \ka_1 \leq \ka_\ga \leq
			\ka_2 $) throughout $ [0,1] $.
	\end{enumerate}
	This set is furnished with the $ C^r $ topology, for some $
	r\geq 2 $.\footnote{The precise value of $ r $ is irrelevant,
		cf.~\lref{L:C^2}.}
\end{dfn}

\begin{dfn}[osculation]\label{D:osculating}
	Two curves $ \ga,\,\eta \colon [0,1] \lto{C^2} S $ on a smooth surface $ S
	$ will be said to \tdfn{osculate} each other at $ t = t_0,\,t_1 \in [0,1] $ if
	one may reparametrize $ \ga $, keeping $ t_0 $ fixed, so that
	\begin{equation*}
		\ga(t_0) = \eta(t_1) \in S,\quad \dot\ga(t_0) = \dot\eta(t_1) \in TS
		\quad \text{and} \quad \ddot\ga(t_0) = \ddot\eta(t_1) \in T(TS).
	\end{equation*}
\end{dfn}

\begin{urmk}\label{R:osculating}
	Suppose that $ S $ is oriented and furnished with a Riemannian metric.  Then
	$ \ga $ and $ \eta $ osculate each other at $ t=t_0,\,t_1 $ if and only if 
	\begin{equation*}
		\ga(t_0) = \eta(t_1),\quad \ta_{\ga}(t_0) = \ta_{\eta}(t_1)\quad
		\text{and} \quad \ka_{\ga}(t_0) = \ka_{\eta}(t_1).
	\end{equation*}
\end{urmk}

\begin{dfn}[circle, hypercircle, horocycle, ray]\label{D:circle}
	A \tdfn{circle} is the locus of all points
	a fixed distance away from a certain point in $ \Hh^2 $ (called its
	\tit{center}). 
	A \tdfn{hypercircle} is one component of the locus of all points
	a fixed distance away from a certain geodesic in $ \Hh^2 $.
	A \tdfn{horocycle} is a curve which meets all geodesics through a
	certain point of $ \bd \Hh^2 $ orthogonally.  
	A \tdfn{ray} is a distance-preserving map $\al\colon
	[0,+\infty)\to \Hh^2 $; 
such a ray is said to \tdfn{emanate from} $ \al'(0) \in UT\Hh^2 $.
\end{dfn}

In order to understand the effect of a geometric transformation on a given
curve, it is often sufficient to replace the latter by its family of osculating
constant-curvature curves. Because the hyperbolic plane is homogeneous and
isotropic, one can represent these in a convenient position in one of the models
as a means of avoiding calculations. 

\begin{figure}[ht]
	\begin{center}
		\includegraphics[scale=.26]{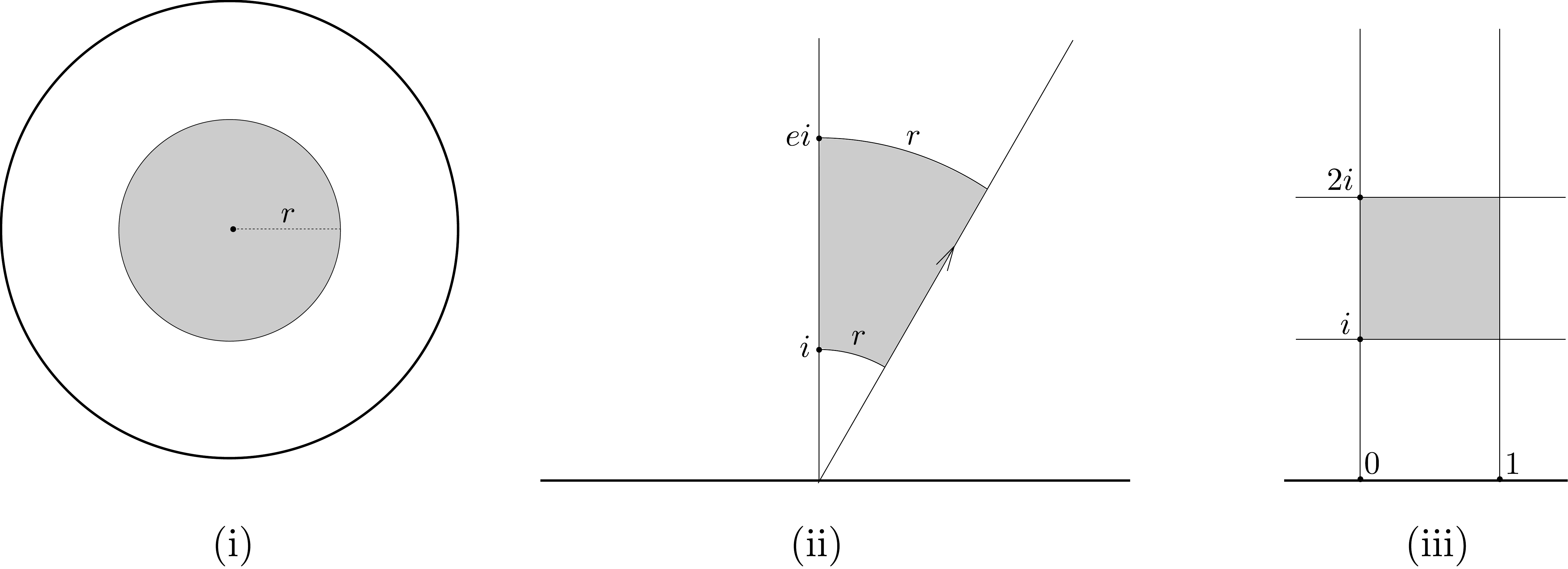}
		\caption{Computing the curvature of circles, hypercircles and horocycles
			through Gauss-Bonnet. The circle in (i) is represented in the disk model,
			while the hypercircle in (ii) and horocycles in (iii) are represented
			in the half-plane model.
	}
		\label{F:gauss_bonnet}
	\end{center}
\end{figure}

\begin{rmk}[constant-curvature curves]\label{R:constant} Define an
	\tdfn{Euclidean circle} to be either a line or a circle in $ \C $; its
	center is taken to be $ \infty $ if it is a line. 
	
	\begin{enumerate}
		\item [(a)] Circles, hypercircles and horocycles are the orbits 
			of elliptic, hyperbolic and parabolic one-parameter subgroups of
			isometries, respectively. In particular, they have constant
			curvature.
		\item [(b)]  In the models $ D $ and $ H $, a circle, horocycle or
			hypercircle appears as the intersection with $ \Hh^2 $ of an
			Euclidean circle which is disjoint, secant or (internally) tangent
			to $ \bd \Hh^2 $, respectively.
		\item [(c)] In the 	hyperboloid model, a circle, hypercircle or horocycle
			appears as a (planar) ellipse, hyperbola or parabola,
			respectively.
		\item [(d)] If all points of a circle lie at a distance $ r > 0 $ from a
			certain point, then its curvature is given by $\pm\coth r $. Thus,
			all circles have curvature greater than $ 1 $ in absolute value. 
		\item [(e)]	If all points of a hypercircle lie at a distance $ r > 0 $ from
			a certain geodesic, then its curvature is given by $ \pm \tanh r $.
			Thus, all hypercircles have curvature less than $ 1 $ in absolute value.
		\item [(f)] The curvature of a horocycle equals $ \pm 1 $.
		\item [(g)] Circles, hypercircles, horocycles and their arcs account for
			all constant-curvature curves.
	\end{enumerate}
	Assertions (d), (e) and (f) may be proved by mapping the curve through an
	isometry to the corresponding curve in Figure \ref{F:gauss_bonnet} and
	applying Gauss-Bonnet to the shaded region. The remaining assertions are
	also straightforward consequences of the transitivity of the group of
	isometries on $ UT\Hh^2 $.
\end{rmk}

In view of the sign ambiguity in the preceding formulas, it is desirable to 
redefine circles, hypercircles and horocycles not as subsets of $ \Hh^2 $, but as 
oriented curves therein. The \tdfn{radius} $ r \in \R $ of a circle or hypercircle is
now defined so that its curvature $ \ka $ is given by
\begin{equation*}
	\ka = \coth r \text{\quad or \quad} \ka = \tanh r,
\end{equation*}
respectively. Then $ \abs{r} $ is the distance from the circle (resp.~hypercircle) to
the point (resp.~geodesic) to which it is equidistant. Both expressions for
the curvature apply to horocycles if these are regarded as circles/hypercircles of
radius $ \pm \infty $, a convention which we shall adopt. Observe that in all
cases the sign of $ r $ is the same as that of the curvature. 

\begin{rmk}[curvature and orientation]\label{R:orientation}
	Let the circle at infinity be oriented from left to right in $ H $ and
	counter-clockwise in $ D $. In either of these models (compare
	\fref{F:constant}): 
	\begin{enumerate}
		\item [(a)] If a hypercircle meets the circle at infinity at an angle $
			\al \in (0,\pi) $, then its curvature equals $ \cos \al $. This 
			follows from a reduction to the hypercircle depicted in Figure
			\ref{F:gauss_bonnet}\?(ii), with the indicated orientation, by
			expressing $ r $ in terms of $ \al $. More explicitly, 
			\begin{equation*}
				r = \int_{\al}^{\frac{\pi}{2}}\frac{1}{\sin t}\,dt 
				= -\log \tan \big( \tfrac{\al}{2} \big) 
				\qquad (\al \in (0,\pi)),
			\end{equation*}
			so that $ \tanh r = \cos \al $.
		\item [(b)] If a circle is oriented (counter-)clockwise, then its
			curvature is less than $ -1 $ (greater than 1). This follows
			immediately from a reduction to Figure
			\ref{F:gauss_bonnet}\?(i).
		\item [(c)] Both preceding assertions can be extended to include
			horocycles. This follows by representing horocycles as circles
			tangent to the circle at infinity in the disk model. 
	\end{enumerate}
\end{rmk}

\begin{figure}[ht]
	\begin{center}
		\includegraphics[scale=.27]{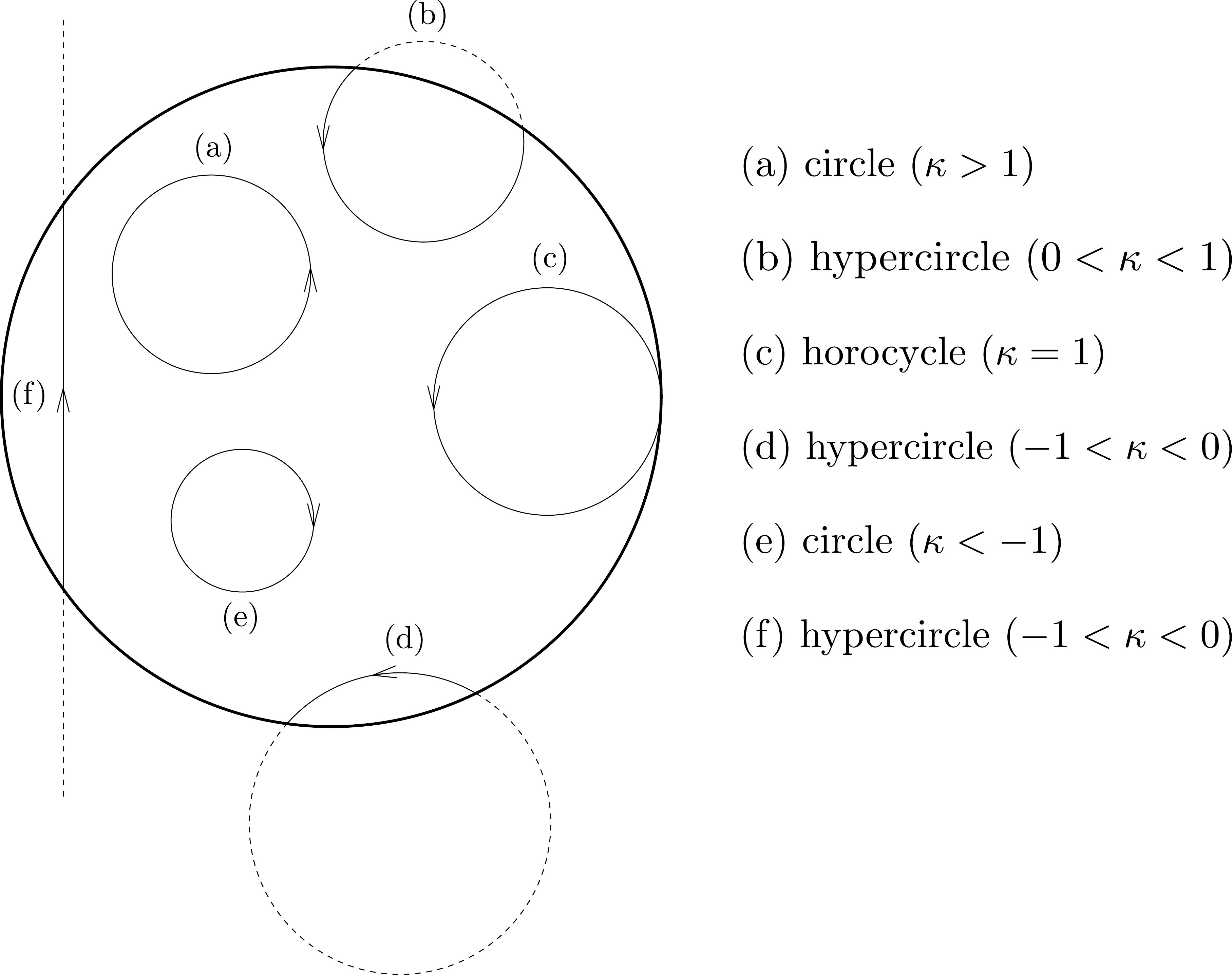}
		\caption{Examples of curves of constant curvature in the disk model.
			Note that the sign of the curvature of a hypercircle need not agree with
		that of its Euclidean curvature.}
		\label{F:constant}
	\end{center}
\end{figure}

\begin{exr}[families of parallel hypercircles]
	Let $ p \neq q \in \Ss^1_\infty $ and $ \ka \in (-1,1) $. Then there exist
	exactly two hypercircles of curvature $ \ka $ through $ p $ and $ q $.
	Furthermore, in the models $ D $ or $ H $:\footnote{Exercises are not used
	anywhere else in the text and may be skipped without any loss.}
	\begin{enumerate}
		\item [(a)] This pair of hypercircles is related by inversion in $
			\Ss^1_\infty $, and their Euclidean centers lie along the
			Euclidean perpendicular bisector $ l $ of $ \ol{pq} $.
	\end{enumerate}
	Let $ o $ and $ o' $ denote the Euclidean centers of $ \Ss^1_\infty $ and the
	geodesic through the pair $ p,\,q $, respectively. 
	\begin{enumerate}
		\item [(b)]  Any point of $ l $ except for $ o $ and $ o' $ is the
			Euclidean center of a unique hypercircle of curvature in $ (0,1) $
			through $ p $ and $ q $.  Describe its orientation in terms of the
			position of its center.
		\item [(c)]	Using \rref{R:orientation}\?(a), describe how its curvature
			changes as its center moves along $ l $.
	\end{enumerate}
\end{exr}


\begin{dfn}[normal translation]\label{D:translation}
	Let $ \ga\colon [0,1]\to \Hh^2 $ be a regular curve and $ \rho\in \R $. The
	\tdfn{normal translation} $ \ga_\rho\colon [0,1]\to \Hh^2 $ of $ \ga $ by 
	$ \rho $ is the  curve given by
	\begin{equation*}
		\ga_\rho(t)=\exp_{\ga(t)}(\rho\no(t))\qquad (t\in [0,1]),
	\end{equation*}
	where $ \no $ is the unit normal to $ \ga $ and $ \exp $ the
	(Riemannian) exponential map. In the hyperboloid model, 
	\begin{equation}\label{E:translation}
		\ga_\rho(t)=\cosh\rho \,\ga(t)+\sinh\rho\,\no(t)\qquad (t\in [0,1]).
	\end{equation}
\end{dfn}

\begin{urmk}\label{R:}
	The unit tangent bundle of any smooth surface has a canonical contact
	structure (cf.~\cite{Thurston}, \S3.7). A curve $ \te{\ga} $ on 
	$ UTS $ is Legendrian (with respect to this structure) if its projection to
	$ S $ always points in the direction prescribed by $ \te{\ga}$. When $ S $
	is oriented, a (local) flow of contact automorphisms $ \phi_{\rho} $ on $ UTS
	$ can be defined by letting $ \phi_{\rho}(u) $ be the parallel translation
	of $ u $ along the geodesic perpendicular to $ u $ by a signed distance of $
	\rho $ toward the left ($ u \in UTS $, $ \rho \in \R $). If $ S $ is
	complete, $ \phi_\rho $ is defined on all of $ UTS $ for each $ \rho \in \R
	$.  The normal translation of $ \ga $ as defined in \dref{D:translation} is
	nothing but the projection to $ S $ of $ \phi_\rho(\ta_\ga) $, in the
	special case where $ S = \Hh^2 $.  As will be seen below, this operation may
	create or remove cusps.
\end{urmk}

\begin{rmk}[normal translation of constant-curvature curves in $ \Hh^2
	$]\label{R:normal}\ Let $ r,\,\rho\in \R $.
	\begin{enumerate} 
		\item [(a)] The normal translation by $ \rho $ of a circle of radius $ r
			$ is a circle of radius $ r-\rho $, equidistant to the same point as
			the original circle. 		
		\item [(b)]  The normal translation by $ \rho $ of a hypercircle of
			radius $ r $ is a hypercircle of radius $ r-\rho $, equidistant to
			the same geodesic as the original hypercircle.  		
		\item [(c)] A normal translation of a horocycle is another horocycle,
			meeting orthogonally the same family of geodesics as the original
			horocycle. 
	\end{enumerate}
	More concisely, the normal translation of a constant-curvature curve of
	radius $ r\in \R\cup\se{\pm\infty}$ by $ \rho\in \R $ is a curve of the same
	type of radius $ r-\rho $. 	
	\begin{enumerate}
		\item [(d)] A normal translation of a hypercircle (resp.~horocycle) meets $
			\bd \Hh^2 $ in the same points (resp.~point) as the original
			hypercircle (resp.~horocycle), when represented in one of the models $ D
			$ or $ H $. 
	\end{enumerate}
	Once again, to prove these assertions one can use an isometry to represent
	the circle (hypercircle, horocycle) as in Figure \ref{F:gauss_bonnet}, where
	they become trivial.  
	
	Notice also that $ \rho $ goes from 0 to $ 2r $, the circle shrinks to a
	singularity ($ \rho = r $) and then expands back to the original circle, but
	with reversed orientation ($ \rho = 2r $).  When $ \rho = r $ the hypercircle
	becomes the geodesic, and when $ \rho = 2r $ it becomes the other component
	of the locus of points at distance $ \abs{r} $ from the geodesic. This
	behavior is subsumed in the following result.
\end{rmk}

\begin{lem}[normal translation of general curves]\label{L:normal2} 
	Let $ \ga\colon [0,1]\to \Hh^2 $ be smooth and regular,
	\begin{equation*}
		\ka_-=\min_{t\in [0,1]}\ka(t)\quad \text{and}\quad 
		\ka_+=\max_{t\in [0,1]}\ka(t).  
	\end{equation*}
	Assume that $ \coth \rho \nin [\ka_-,\ka_+] $. Then:
	\begin{enumerate}
		\item [(a)] The normal translation $ \ga_{\rho} $ is regular. In
			particular, $ \ga_\rho $ is regular for all $ \rho $ in some open
			interval containing 0, and for all $ \rho \in \R $ in case $
			[\ka_-,\ka_+]\subs [-1,1] $.
		\item [(b)] $ \ta_{\ga_\rho} \equiv \ta_{\ga} $ if these are regarded as
			taking values in $ \E^{2,1} \sups L $. 
	\end{enumerate}
	Given $ t\in [0,1] $, there exists a unique constant-curvature curve which
	osculates $ \ga $ at $ \ga(t) $. The \tdfn{radius of curvature} $ r_\ga(t) $
	of $ \ga $ at $ \ga(t) $ is defined as the radius of this osculating curve.
	\begin{enumerate}
		\item [(c)]	The radii of curvature of $ \ga_\rho $ and $ \ga $ are
			related by $ r_{\ga_\rho} = r_\ga - \rho $.
		\item [(d)] The curvature of $
			\ga_\rho $ is given by:
			\begin{equation*}
				\ka_{\ga_\rho}(t)=\begin{cases}
					\frac{1-\ka_\ga(t)\coth\rho}{\ka_\ga(t) - \coth \rho} &
					\text{if\quad $ \vert{\ka_\ga(t)}\vert>1 $;} \\
					\frac{\ka_\ga(t) - \tanh \rho}{1-\ka_\ga(t)\tanh\rho} &
					\text{if\quad  $ \vert{\ka_\ga(t)}\vert<1 $;} \\
					\ka_\ga(t) & \text{if\quad $ \vert{\ka_\ga(t)}\vert=1 $. } 
				\end{cases}
			\end{equation*}
		\item [(e)] $ (\ga_\rho)_{-\rho}=\ga $.
	\end{enumerate}
\end{lem}
\begin{proof}
	For (a) and (b), use \eqref{E:translation}. It is clear from the
	definition of ``osculation'' that if $ \eta $ osculates $ \ga $ at $ \ga(t)$,
	then $ \eta_\rho $ osculates $ \ga_\rho $ at $ \ga_\rho(t) $. Thus part (c)
	is a consequence of \eref{R:normal}.  Part (d) follows from the addition
	formulas for $ \coth $ and $ \tanh $, and part (e) is obvious.
\end{proof}

The topology of $ \sr C_{\ka_1}^{\ka_2}(u,v) $ depends in principle upon eight
real parameters: three for each of $ u,\,v\in UT\Hh^2 $, and two for the curvature
bounds. This number can be halved by a suitable use of normal translations and
isometries.  In the sequel, two intervals are said to \tdfn{overlap} if they
intersect but neither is contained in the other one.

\begin{prp}[parameter reduction]\label{P:reduction}
	Let $ (\ka_1,\ka_2) \neq (-1,1) $ and $ u,\,v,\,\bar{u} \in UT\Hh^2 $ be given.
	Then there exist $ \bar{v}\in UT\Hh^2 $ and $ \ka_0 $ such that $ \sr
	C_{\ka_1}^{\ka_2}(u,v) $ is canonically homeomorphic to a space of the type
	listed in Table \ref{Ta:reduction}. 
\end{prp}

\begin{table}[h!]
	\begin{center}
		\begin{tabular}{ c c c }\hline 
			Case & $ \sr C_{\ka_1}^{\ka_2}(u,v) $ homeomorphic to & Range of
			$ \ka_0 $ \rule[-8pt]{0pt}{22pt}  \\ \hline

			$ (\ka_1,\ka_2) $ contained in $ [-1,1] $ & $ \sr
		C_{0}^{\ka_0}(\bar{u},\bar{v}) $ & $ (0,1] $  \rule{0pt}{14pt} \\	

		$ (\ka_1,\ka_2) $ disjoint from $ [-1,1] $ & $ \sr
		C_{\ka_0}^{+\infty}(\bar{u},\bar{v}) $ & $ [1,+\infty) $
			\rule{0pt}{14pt}	\\

			$ (\ka_1,\ka_2) $ overlaps $ [-1,1] $ &
			$ \sr C_{\ka_0}^{+\infty}(\bar{u},\bar{v}) $ &  $ [-1,1) $
				\rule{0pt}{14pt} \\	

				$ (\ka_1,\ka_2) $ contains $ [-1,1] $ & $ \sr
			C_{-\ka_0}^{+\ka_0}(\bar{u},\bar{v}) $ & $ (1,+\infty] $   
			\rule{0pt}{14pt} \\	

		\end{tabular}\vspace{10pt}
		\caption{
			Reduction of the number of parameters controlling the topology of 
			$ \sr C_{\ka_1}^{\ka_2}(u,v) $, for $ (\ka_1,\ka_2) \neq (-1,1) $. 
		}
		\label{Ta:reduction} 
	\end{center}
\end{table}

The special role played by the interval $ [-1,1] $ stems from the fact that any
normal translation of a horocycle is a horocycle.   The four classes listed
in Table \ref{Ta:reduction} are genuinely different in terms of their
topological properties, as discussed in the introduction. In the only case not
covered by \pref{P:reduction}, namely $ (\ka_1,\ka_2)=(-1,1) $, given $
u,\,v,\,\bar{u} $, one can find $ \bar v $ such that $ \sr C_{-1}^{+1}(u,v) $ is
homeomorphic to $ \sr C_{-1}^{+1}(\bar u,\bar v) $;  for this it is sufficient
to compose all curves in the former with an isometry taking $ u $ to $ \bar{u}
$. 

\begin{proof}
	The argument is similar for all four classes, so we consider in detail only
	the case where $ (\ka_1,\ka_2) $ overlaps $ [-1,1] $.  Firstly, notice that
	composition with an orientation-reversing isometry switches the sign of the
	curvature of a curve. Therefore, by applying a reflection in some geodesic
	if necessary, it can be assumed that $ -1 \leq \ka_1 < 1 < \ka_2 $ (instead
	of $ \ka_1 < -1 < \ka_2 \leq 1 $).  Write $ \ka_2 = \coth \rho_2 $, $ \ka_1
	= \tanh \rho_1 $ and $ \ka_0 = \tanh(\rho_1-\rho_2) $. Then normal
	translation by $ \rho_2 $ (i.e., the map $ \ga \mapsto \ga_{\rho_2} $)
	provides a homeomorphism
	\begin{equation*}
		\sr C_{\ka_1}^{\ka_2}(u,v)\to \sr C_{\ka_0}^{+\infty}(u',v'),
	\end{equation*}
	where $ u' $ is obtained by parallel translating $ u $ by a distance $
	\rho_2 $ along the ray emanating from $ Ju $, and similarly for $ v'
	$. The inverse of this map is simply normal translation by $ -\rho_2 $. All
	details requiring verification here  were already dealt with in 
	\lref{L:normal2}.  Finally, post-composition of curves with the unique
	orientation-preserving hyperbolic isometry taking $ u' $ to $ \bar{u} $
	yields a homeomorphism 
	\begin{equation*}
		\sr C_{\ka_0}^{+\infty}(u',v') \to
		\sr C_{\ka_0}^{+\infty}(\bar{u},\bar{v}).
	\end{equation*}
	
	In the remaining cases, let $ \rho_i $ denote the radius of a curve of
	constant curvature $ \ka_i $ ($ i=1,2 $). For the first class in the table,
	apply a reflection if necessary to ensure that $ \ka_1 > -1 $, then use
	normal translation by $ \rho_1 $. For the second class, reduce to the case
	where $ \ka_1 \geq 1 $ and apply normal translation by $ \rho_2 $. For the
	fourth class, apply a normal translation by $ \frac{1}{2}(\rho_1+\rho_2) $.
	In all cases, use an orientation-preserving isometry to adjust the initial
	unit tangent vector to $ \bar u $. 
\end{proof}

\begin{urmk}
	The homeomorphism constructed in the proof of \pref{P:reduction} operates on
	the curves in a given space by a composition of a normal translation and an
	isometry. It is ``canonical'' in the sense that these two transformations,
	as well as the values of $ \ka_0 $ and $ \bar v $, are uniquely determined
	by $ \ka_1,\,\ka_2,\,u,\,v $ and $ \bar u $.  However, this does not
	preclude the existence of some more complicated homeomorphism between spaces
	in the second column of Table \ref{Ta:reduction}.
\end{urmk}

\begin{exr}[extension of \pref{P:reduction}]\label{X:reduction}
	Let $ (\ka_1,\ka_2) \neq (-1,1) $ and $ u,\,v,\,\bar u \in UT\Hh^2 $ be
	given. Use the argument above to prove:
	\begin{enumerate}
		\item [(a)] If $ (\ka_1,\ka_2) $ contains $ [-1,1] $, then there exist $
			\ka_0 \in [-\infty,-1) $ and $ \bar v \in UT\Hh^2 $ such that
			\begin{equation*}
				\sr C_{\ka_1 }^{\ka_2 }(u,v) \home 
				\sr C_{\ka_0}^{+\infty}(\bar{u},\bar{v}). 
			\end{equation*}
		\item [(b)] If $ -1<\ka_1 < \ka_2 < 1 $, then there exist $ \ka_0 \in
			(0,1) $ and $ \bar v \in UT\Hh^2 $ such that
			\begin{equation*}
				\sr C_{\ka_1}^{\ka_2}(u,v) \home \sr C_{-\ka_0}^{+\ka_0}(\bar
				u,\bar v). 
			\end{equation*}
			Note that the hypothesis here is more restrictive than in the 
			first case of the table.
	\end{enumerate}
\end{exr}


\section{Voidness of the canonical subspaces}\label{S:voidness}

The purpose of this section is to discuss which of the canonical subspaces of $
\sr C_{\ka_1}^{\ka_2}(u,v) $ are empty. In the sequel $ \sr
C_{\ka_1 }^{\ka_2 }(u,\cdot) $ (resp.~$ \bar{\sr C}_{\ka_1 }^{\ka_2 }(u,\cdot) $)
denotes the space obtained from \dref{D:spaces} by omitting the restriction that
$ \ta_{\ga}(1) = v $.

\begin{lem}[attainable endpoints]\label{L:attainable}
	If at least one $ \abs{\ka_i} > 1 $, then $ \sr
	C_{\ka_1}^{\ka_2}(u,v) \neq \emptyset $ for all $ u,v\in UT\Hh^2 $.
\end{lem}

\begin{proof}
	By symmetry, we may assume that $ \ka_2>1 $. Further, using normal
	translation by $ \rho_2=\arccoth \ka_2 $, we reduce to the case where $
	\ka_2=+\infty $.  Let $ q \in \Hh^2 $ and $ x \in UT\Hh^2_q $ be arbitrary.
	Consider the map 
	\begin{equation*}
		F_x \colon \sr C_{\ka_1 }^{+\infty}(x,\cdot) \to UT\Hh^2,\quad F_x(\ga) =
		\ta_{\ga}(1). 
	\end{equation*}
	It is not hard to see that this map is open; cf.~ \lref{L:submersion}. 
	We begin by proving that $ \Im(F_x) \sups UT\Hh^2_q $.  
	Since sufficiently tight circles are allowed, $ x \in \Im(F_x) $.
	Let $ I\subs UT\Hh^2_q  \cap \Im(F_x) $ be an open interval about $ x $.
	For each $ y\in I $, let $ R_y $ denote the elliptic isometry
	fixing $ q $ and taking $ x $ to $ y $. Let
	\begin{equation*}
		\phantom{\quad (y,\,z \in UT\Hh^2_q)} 
		\ga_y \in \sr C_{\ka_1 }^{\ka_2 }(x,y) \quad 
		\text{and} \quad \ga_z \in \sr C_{\ka_1 }^{\ka_2 }(x,z)
		\quad (y,\,z \in UT\Hh^2_q).
	\end{equation*}
	Then the concatenation 
	\begin{equation*}
		\ga_y\ast \big( R_y\circ \ga_z \big)
	\end{equation*}
	starts at $ x $ and ends at $ dR_y(z) $. This curve may not have a second
	derivative at the point of concatenation, but it may be approximated by a
	smooth curve in $ \sr C_{\ka_1 }^{\ka_2 }(x,R_yz) $. This shows that $
	\Im(F_x) $ contains the interval $ \set{R_yz}{y \in I} $ about $ z $
	whenever it contains $ z $.  In turn, this implies that $ \Im(F_x) $
	includes $ UT\Hh^2_q $.  Therefore, if $ \Im(F_u) $ contains any tangent
	vector $ x $ based at $ q $, it must also contain $ UT\Hh^2_q \subs \Im(F_x)
	$, by transitivity. 
	
	Now let $ r:=\arccoth \ka_1 $ in case $ \ka_1>1 $ and $ r:=+\infty $
	otherwise. If $ \Im(F_u) $ contains tangent vectors at $ q $, then it
	also contains tangent vectors at $ q' $ for any $ q'\in B(q;2r) $, since
	$ q' $ can be reached from $ q $ by the half-circle centered at the midpoint
	of the geodesic segment $ qq' $ of radius $ \frac{1}{2}d(q,q') <r$. We
	conclude that $ \Im(F_u)=UT\Hh^2 $, that is, $ \sr C_{\ka_1 }^{\ka_2 }(u,v)
	\neq \emptyset $ for all $ v \in UT\Hh^2 $.
\end{proof}

The converse of \lref{L:attainable} is an immediate consequence of the following
result; cf.~also \cref{C:closed}. 

\begin{lem}[unattainable endpoints]\label{L:unattainable}
	Let $ (\ka_1,\ka_2)\subs[-1,1] $ and $ u\in UT\Hh^2 $. The hypercircles (or
	horocycles) of curvature $ \ka_1 $ and $ \ka_2 $ tangent to $ u $ divide $
	\Hh^2 $ into four open regions.  Any curve in $ \sr
	C_{\ka_1}^{\ka_2}(u,\cdot) $ is confined to one of these regions $ R $.
	Conversely, a point in $ R $ can be reached by a curve of constant curvature
	in $ \sr C_{\ka_1}^{\ka_2}(u,\cdot) $.
\end{lem}
\begin{proof}
	The hypercircle of curvature $ \frac{1}{2}(\ka_1+\ka_2) $ tangent to $ u $ meets
	$ \bd \Hh^2 $ at two points; let $ R $  denote the open region whose closure
	contains the one point in $ \bd \Hh^2 $ towards which $ u $ is pointing. Let
	$ \ga \in \sr C_{\ka_1 }^{\ka_2 }(u,\cdot) $. Comparison of the curvatures
	of $ \ga $ and those of the hypercircles bounding $ R $ shows that $ \ga(t) \in R
	$ for sufficiently small $ t>0 $.  Suppose for a contradiction that $ \ga $
	reaches $ \bd R $ for the first time at $ t=T \in (0,1] $ at some point of the
	hypercircle $ E $ of curvature $ \ka_2 = \tanh r $. (In case $ \ka_2=1 $, $ E $
	is a horocycle and $ r=+\infty $.)  The function 
	\begin{equation*}
		\de\colon [0,T] \to \R ,\quad  \de(t)=d(\ga(t),E) 
	\end{equation*}
	vanishes at $ 0 $ and $ T $, hence it must attain its global maximum in $
	[0,T] $ at some $ \tau \in (0,T) $, say $ \de(\tau) = \rho $. By
	\rref{R:normal}, the normal translation of $ E $ by $ -\rho $ is
	a hypercircle (resp.~horocycle) of curvature $ \tanh (r+\rho) $, which
	must be tangent to $ \ga $ at $ \ga(\tau) $ because $ \dot \de(\tau) = 0 $.
	Since $ \de $ has a local maximum at $ \tau $, comparison of curvatures
	yields that $ \ka_{\ga}(\tau)\geq \tanh(r+\rho)>\ka_2 $, which is
	impossible.

	The last assertion of the lemma holds because constant-curvature curves
	in $ \sr C_{\ka_1}^{\ka_2}(u,\cdot) $ foliate $ R $. 
\end{proof}

\begin{urmk}\label{R:Rbar}
	With the notation of \lref{L:unattainable}, the image of any
	curve in $ \bar{\sr C}_{\ka_1}^{\ka_2}(u,\cdot) $ is contained in $ \bar
	R $ (see \dref{D:spaces} for the definition of $ \bar {\sr C} $). 
	To prove this, let $ \ga $ be a curve in this space. If $ \ka_\ga \equiv
	\ka_1 $ or $ \ka_\ga\equiv \ka_2 $, then the image of $ \ga $ is entirely
	contained in $ \bd R $. Otherwise, let $ t_0 < 1 $ be the infimum of all $ t
	\in [0,1] $ such that $ \ka_{\ga}(t) \in (\ka_1,\ka_2) $, and apply the argument
	of \lref{L:unattainable} to $ \ga|_{[t_0,1]} $. 
\end{urmk}

\begin{crl}\label{C:closed}
	The spaces $ \sr C_{\ka_1 }^{\ka_2 }(\cdot,\cdot) $ and $ \bar{\sr
	C}_{\ka_1}^{\ka_2}(\cdot,\cdot) $ contain closed curves if and only if at least
	one $ \abs{\ka_i}>1 $.\qed
\end{crl}

\begin{crl}\label{C:nontrivial}
Let $ S $ be a hyperbolic surface. Then any closed curve whose curvature is
bounded by 1 in absolute value is homotopically non-trivial (that is, it
cannot represent the unit element of $ \pi_1S $).
\end{crl}
\begin{proof}
Indeed, the previous corollary shows that the lift of such a curve to $
\Hh^2 $ cannot be closed.
\end{proof}

It will be convenient to introduce another model for $ \Hh^2 $, which bears some
similarity to Mercator world maps.

\begin{dfn}[Mercator model]\label{D:Mercator}
The underlying set of the \tdfn{Mercator model} $ M $ is $ (0,\pi) \times \R
$. Its metric is defined by declaring the correspondence
\begin{equation*}
	M\to H,\quad (x,y)\mapsto e^{y+ix} \qquad \big(x\in (0,\pi),~y\in \R) 
\end{equation*} 
to be an isometry. Because this map is the composition of the complex
exponential with a reflection in the line $ y=x $, $ M $ is conformal.
\end{dfn}

\begin{rmk}[geometry of $ M $]\label{R:compM} It is straightforward to verify
	that in the Mercator model $ M $:
\begin{enumerate}
	\item [(a)] Vertical lines $ y\mapsto (x,y) $, or
		{parallels}, represent hypercircles of curvature $ \cos x $,
		corresponding in $ H $ to rays having $ 0 \in H $ for their initial point
		(see \xref{R:orientation}\?(a)).  Horizontal segments, or
		{meridians}, are geodesics corresponding in $ H $ to Euclidean
		half-circles centered at 0. 
	\item [(b)] Vertical translations are hyperbolic isometries. The
		Riemannian metric $ g $ is given by
		\begin{equation*}
			g_{(x,y)} = \frac{dx^2+dy^2}{\sin^2 x} \qquad  ((x,y)\in M).
		\end{equation*}
	\item [(c)] The Christoffel symbols are given by
		\begin{equation*}
			\Ga_{ij}^{k}(x,y) = \begin{cases}
				0 & \text{ if $ i+j+k $ is odd;} \\
				(-1)^{1+ij}\cot x & \text{ if $ i+j+k $ is even.}
			\end{cases}
		\end{equation*}
\end{enumerate}
\end{rmk}

\begin{lem}\label{L:graph}
Let $ (\ka_1,\ka_2) \subs [-1,1] $  and $ u $ be the vector $ 1 \in \Ss^1 $
based at $ \big(\frac{\pi}{2},0\big) \in M $. Then the image of any curve
in $ \sr C_{\ka_1 }^{\ka_2 }(u,\cdot) $ or $ \bar{\sr
C}_{\ka_1}^{\ka_2}(u,\cdot) $ is the graph of a function of $ x $ when
represented in $ M $. Conversely, if $ (\ka_1,\ka_2) \not \subs [-1,1] $,
then there exist curves in these spaces which are not graphs.
\end{lem}

\begin{proof}
Suppose that $ \ga \in \bar{\sr C}_{\ka_1}^{\ka_2}(u,\cdot) $ is not the
graph of function of $ x $ when represented in $ M $, i.e., it is tangent to
a parallel $ P_1 $ at some time, say, $ t=1$.  Let $ P_0 $ be the parallel $
x=\frac{\pi}{2} $, which is orthogonal to $ \ga $ at $ t=0 $ by hypothesis; note
that $ P_0 $ is a geodesic. Let $ L $ be the meridian through $ \ga(1) $, which
is a geodesic orthogonal to both $ P_0 $ and $ P_1 $.  Let $ \ga_1 $ be the
concatenation of $ \ga $ and its reflection in $ L $ (with reversed
orientation). Let $ \ga_2 $ be the concatenation of $ \ga_1 $ and its reflection
in $ P_0 $ (again with reversed orientation).  Then $ \ga_2 $ is a closed curve,
hence $ (\ka_1,\ka_2) \not\subs [-1,1] $ by \cref{C:closed}. 

Conversely, if $ (\ka_1,\ka_2) \not\subs [-1,1] $, then $ \sr
C_{\ka_1}^{\ka_2}(u,\cdot) $ contains circles, which are closed and hence
not graphs.  
\end{proof}

\begin{prp}[attainable turnings]\label{P:unattainable}
Consider the decomposition \eqref{E:decomp} of 
$ \sr C_{\ka_1 }^{\ka_2 }(u,v) $ into canonical subspaces.
\begin{enumerate}
	\item [(a)] If $ (\ka_1,\ka_2) $ contains $ [-1,1] $, then all of its 
		canonical subspaces are nonempty.
	\item [(b)] If $ (\ka_1,\ka_2) $ is contained in $ [-1,1] $, then at
		most one canonical subspace is nonempty.
	\item [(c)] If $ (\ka_1,\ka_2) $ overlaps or is disjoint from $ [-1,1]
		$, then infinitely many of the canonical subspaces are nonempty, and
		infinitely many are empty. 
\end{enumerate}
\end{prp}
\begin{proof}
We split the proof into parts.
\begin{enumerate}
	\item [(a)] By \lref{L:attainable}, $ \sr C_{\ka_1 }^{\ka_2 }(u,v) \neq
		\emptyset $. Since $ (\ka_1,\ka_2) \sups [-1,1] $, we may concatenate
		a curve in this space with circles of positive or negative
		curvature, traversed multiple times, to attain any desired total
		turning.
	\item [(b)] By \pref{P:reduction}, we may assume that $ u $ is the
		vector in the statement of \lref{L:graph}. Then the assertion 
		becomes an immediate consequence of the latter.
	\item [(c)] By \lref{L:attainable}, $ \sr C_{\ka_1}^{\ka_2}(u,v) \neq
		\emptyset $. Because we may concatenate any curve in $ \sr C_{\ka_1
		}^{\ka_2 }(u,v) $ with a circle (of positive curvature if $ \ka_1 >
		-1 $ and of negative curvature if $ \ka_2 < 1 $) traversed multiple
		times, $ \sr C_{\ka_1 }^{\ka_2 }(u,v;\tau) \neq \emptyset $ for
		infinitely many values of $ \tau $. The remaining
		assertion, that $ \sr C_{\ka_1 }^{\ka_2 }(u,v;\tau) = \emptyset $ for
		infinitely many values of $ \tau $, is a consequence of
		\lref{L:turnings}\?(c) below.\qedhere
\end{enumerate}
\end{proof}

\begin{dfn}[$ \al_{\pm} $]\label{D:arcs}
	Given a regular curve $ \ga \colon [0,1] \to \Hh^2 $, define maps 
	\begin{equation*}
		\al_{\pm} \colon [0,1] \to \Ss^1_\infty 
	\end{equation*}
	by letting $ \al_{\pm}(t) $ be the point where the geodesic ray emanating
	from $ \pm \no(t) $ meets $ \Ss^1_\infty $.  
\end{dfn}

\begin{lem}\label{L:turnings}
	Let $ u,\,v\in UT\Hh^2 $ be fixed.
	\begin{enumerate}
		\item [(a)] Two curves $ \ga,\,\bar\ga \in \sr
			C_{-\infty}^{+\infty}(u,v) $ lie in the same component of this space
			if and only if the associated maps $ \al_{+},\,\bar\al_{+} \colon
			[0,1] \to \Ss^1_\infty $ defined in \dref{D:arcs} have the same
			total turning. 
		\item [(b)] If $ \ka_\ga > -1 $ everywhere, then $ \al_- $ is monotone.
			Similarly, if $ \ka_\ga<+1 $, then $ \al_{+} $ is monotone.
		\item [(c)] If $ \ka_1 \geq -1 $, then then there exists $ \tau_0 $ such
			that $ \sr C_{\ka_1 }^{\ka_2 }(u,v;\tau) $ is empty for all $ \tau
			<\tau_0 $.
	\end{enumerate}
\end{lem}

\begin{proof}
	It is clear that if $ \ga,\,\bar\ga $ lie in the same component, then $
	\al_{+} \iso \bar\al_{+} $ and $ \al_{-} \iso \bar\al_{-} $. Conversely, if
	$ \ga,\,\bar\ga $ do not lie in the same component, then $ \bar\ga $ must be
	homotopic (through regular curves) to the concatenation of $ \ga $ with a
	circle traversed $ n $ times, for some $ n\neq 0 $. This yields a homotopy
	between $ \bar\al_+ $ and $ \al_+ $ concatenated with a map of degree $ n $.
	
	For part (b), it suffices to approximate $ \ga $ by its osculating
	constant-curvature curve at each point. If such a curve is a circle, place
	its center at the origin in the disk model. If it is a hypercircle, regard it as
	an Euclidean ray in $ H $. 

	Finally, (c) is a corollary of (a) and (b).
\end{proof}

\begin{exr}\label{X:arcs} Let $ \ga \colon [0,1] \to \Hh^2 $ be a regular curve
	and $ A_{\pm}(\ga) \subs \Ss^1_\infty $ denote the images of $ \al_{\pm}
	\colon [0,1] \to \Ss^1_\infty $.  
	\begin{enumerate}
		\item [(a)] Each of $ A_{\pm}(\ga) $ is a closed arc, which may
			be a singleton or all of $ \Ss^1_\infty $. 
		\item [(b)] Regard $ \ga $ as a curve in $ H $, and let 
			$ \infty $ be the unique point of $ \Ss^1_\infty $ not
			on the real line in this model. Then $ \infty $ lies in the
			complement of $ A_{+}(\ga) \cup A_{-}(\ga)	$ if and only if $
			\ta_\ga $ is never horizontal. 
		\item [(c)] Let $ \eta $ be a horocycle of curvature $ 1 $,
			tangent to $ \Ss^1_\infty $ at $ z $. Then $ A_+(\eta) = \se{z} $
			while $ A_-(\eta) = \Ss^1_\infty \ssm \se{z} $. 
			What happens if the curvature of $ \eta $ is $ -1 $? (\tit{Hint}:
			reduce to the case where $ z = \infty $.)
	\end{enumerate}
\end{exr}


\section{Frame and logarithmic derivative}\label{S:frame}

The group $ \Iso_+(\Hh^2) $ of all orientation-preserving isometries of the
hyperbolic plane acts simply transitively on $ UT\Hh^2 $. Therefore, an element
$ g$ of this group is uniquely determined by where it maps a fixed unit tangent
vector $ u_0 $. This yields a correspondence between the two sets, viz., $ g\dar
gu_0 $.  The \tdfn{frame} 
\begin{equation}\label{E:frame}
	\Phi=\Phi_\ga\colon [0,1]\to \Iso_+(\Hh^2) 
\end{equation}
of a regular curve $ \ga\colon [0,1]\to \Hh^2 $ is the image of $ \ta_\ga $
under this correspondence, and the \tdfn{logarithmic derivative} $
\La=\La_\ga\colon [0,1]\to \textrm{L}(\Iso(\Hh^2)) $ is its infinitesimal
version. More precisely, $ \La $ is the translation of $ \dot\Phi $ to the Lie
algebra $ \textrm{L}(\Iso(\Hh^2)) $, defined by
\begin{equation}\label{E:La}
	\dot\Phi(t) = TL^{\Phi(t)} (\La(t)) \quad (t\in [0,1]).
\end{equation}
Here $ TL^{\Phi(t)} $ denotes the derivative (at the identity) of left
multiplication by $ \Phi(t) $. Although $ \Phi $ depends on the choice
of $ u_0 $, $ \La $ does not.

To be explicit, let $ S=\diag(-1,1,1) $, 
\begin{equation*}
	\Oo_{2,1} = \set{ Q \in \GL_3(\R)}{Q^tSQ=S}
\end{equation*}
be the group of isometries of $ \E^{2,1} $ and 
\begin{equation*}
	\SO^+_{2,1} = \set{Q\in \Oo_{2,1} }{\det(Q)=1\text{ and }Q(L)=L} 
\end{equation*}
be the connected component of the identity, which is isomorphic to $
\Iso_+(\Hh^2) $.  The corresponding Lie algebra is
\begin{equation*}
	\aso_{2,1} = \set{X\in \agl_3(\R)}{X^tS+SX=0}.
\end{equation*}
In the hyperboloid model $ L $, the identification between $
UT\Hh^2 $ and $ \Iso_+(\Hh^2) $ takes the canonical form 
\begin{equation}\label{E:corresp}
	UTL_p\ni u \dar
	\begin{pmatrix}
		| & | & | \\
		p & u & p \ten u \\
		| & | & |
	\end{pmatrix}	\in \SO^+_{2,1},
\end{equation}
where $ \ten $ denotes the Lorentzian vector product in $ \E^{2,1}
$. (Recall that $ u \ten v = S(u \times v)$, where $ \times $ denotes the usual
vector product of vectors in $ \R^3 $.) Note that this correspondence is also of
the form $ gu_0 \dar g $ described above ($ g \in \SO^+_{2,1} $), for $ u_0 =
e_1 = (0,1,0) \in UTL $.

\begin{exr}[computations in $ L $]\label{X:loid}
	Let $ \ga\colon [0,1]\to L $ be smooth and regular. 
	\begin{enumerate}
		\item [(a)] Denoting differentiation with respect to the given parameter
			(resp.~arc-length) by $ \dot{} $ (resp.~$ ' $):
			\begin{equation*}
				\ta' = \ka\no+\ga,\quad \no'=-\ka\ta\quad \text{and} \quad
				\ka = \ta'\cdot \no = \frac{1}{\norm{\dot\ga}}\dot\ta\cdot
				\no = \frac{1}{\norm{\dot\ga}^2}\ddot\ga\cdot\no =
				\frac{1}{\norm{\dot\ga}^3}\det(\ga,\dot
				\ga,\ddot \ga).
			\end{equation*}
			In these formulas $ \ga,\,\ta,\,\no $ are viewed as taking values in
			$ \E^{2,1} $, $ \cdot $ is the Lorentzian inner product and $
			\norm{\ }^2 $ the corresponding quadratic form.
		\item [(b)] The frame $ \Phi\colon [0,1]\to \SO^+_{2,1} $ and 
			logarithmic derivative $ \La\colon [0,1]\to \aso_{2,1} $ of $ \ga $
			are given by
			\begin{equation*}
				\Phi=\begin{pmatrix}
					| & | & | \\
					\ga & \ta & \no \\
					| & | & |
				\end{pmatrix}
				\text{\quad and\quad}
				\La = \norm{\dot\ga}\begin{pmatrix}
					0 & 1 & 0 \\
					1 & 0 & -\ka \\
					0 & \ka & 0
				\end{pmatrix}.
			\end{equation*}
	\end{enumerate}
\end{exr}

Similar formulas for the curvature in the disk and half-plane models are
much more cumbersome. However, the frame and logarithmic derivative 
do admit comparatively simple expressions. 

For concreteness, we choose the correspondence between $ \PSL_2(\R)	= \Iso_+(H)
$ and $ UTH $ to be $ M \dar dM_i(1) $, where the latter denotes the image
under the complex derivative $ dM $ of the tangent vector $ 1 \in \C $ based at
$ i\in H $. Similarly, we choose the correspondence between $ \Iso_+(D) $ and $
UTD $ to be  $ M \dar dM_0(\frac{1}{2}) $. Recall that $ \Iso_+(D) $ consists of
those M\"obius transformations of the form
\begin{equation*}
	\qquad z\mapsto \frac{az+b}{\bar bz+\bar a},\quad \abs{a}^2-\abs{b}^2>0
	\quad (a,\,b \in \C).
\end{equation*}

\begin{rmk}[$ \Phi $ and $\La $ in the models $ D $ and $ H $]\label{R:phila}
	Denote elements of projective groups as matrices in square brackets, the
	absolute value of a complex number by $ \abs{\ } $ and, exceptionally,
	the norm of a vector tangent to $ \Hh^2 $ by $ \norm{\ } $. 
	\begin{enumerate} 
		\item [(a)] Let $ \ga\colon [0,1] \to H $ be a smooth regular curve.
			Then $ \La \colon [0,1]\to \asl_2({\R}) $ is given by
			\begin{equation*}
				\La=\frac{1}{2}\norm{\dot\ga}
				\begin{pmatrix}
					0 & 1+\ka \\
					1-\ka & 0 
				\end{pmatrix}
			\end{equation*}
			and $ \Phi \colon [0,1]\to \PSL_2(\R) $ is given by 
			\begin{equation*}
				\Phi=
				\begin{bmatrix}
					\Im(\ga \bar z) & \Re(\ga \bar z)\\
					\Im(\bar z) & \Re(\bar z)
				\end{bmatrix},\quad
				\text{where $ \frac{\ta}{\abs{\ta}} = z^2\in \Ss^1 $.}
			\end{equation*}
		\item [(b)] Let $ \ga\colon [0,1] \to D $ be a smooth regular curve.
			Then $ \La\colon [0,1] \to \asl_2(\C) $ is given by 
			\begin{equation*}
				\La=\frac{1}{2}\norm{\dot\ga}\begin{pmatrix}
					i\ka & 1 \\
					1 & -i\ka 
				\end{pmatrix}
			\end{equation*}
			and $ \Phi\colon [0,1] \to \Iso_+(D)\subs\PSL_2(\C) $ is given by
			\begin{equation*}
				\Phi=
				\begin{bmatrix}
					\phantom{\ga}z & \ga \bar z\\
					\bar\ga z & \phantom{\ga}\bar z
				\end{bmatrix},\quad 
				\text{where $ \frac{\ta}{\abs{\ta}}=z^2\in \Ss^1 $.}
			\end{equation*}
	\end{enumerate}
	To establish the first formula, it suffices by \eqref{E:La} to consider the
	case the parameter is arc-length and $ \Phi(s_0) = I\in \PSL_2(\R) $.
	Without trying to find an expression for $ \Phi(s) $ itself, write
	\begin{equation*}
		M^s:=\Phi(s)\in \PSL_2(\R),\quad
		M^s(z)=\frac{a(s)z+b(s)}{c(s)z+d(s)}, \quad M^{s}(i)=\ga(s)\in H, \quad
		dM^{s}_i=\ta(s)\in \C.  
	\end{equation*} 
	(Here $ dM^s $ denotes the complex derivative of $ M^s $ as a map $ \C\to \C
	$.) Differentiate with respect to $ s $ at $ s_0 $ and express $ \ta'(s_0) $
	in terms of $ \frac{D\ta}{ds}(s_0) $ to determine $ a',b',c',d' $ at $ s_0
	$. The derivation of the formula for $ \La $ in (b) is analogous, and the
	expressions for $ \Phi $ are obtained by straightforward computations.
\end{rmk}

A one-parameter group of hyperbolic isometries provides a foliation of $
\Hh^2 $ by its orbits, which are hypercircles by \rref{R:constant}\?(a). A
regular curve $ \ga \colon [0,1]\to \Hh^2$ \tdfn{admits hyperbolic grafting} if
there exist some such group $ G $ and $ t_0,t_1\in [0,1] $ such that $
\ta(t_0) $ and $ -\ta(t_1) $ are tangent to orbits of $ G $.

\begin{rmk}[hyperbolic grafting]\label{R:grafting}
	A regular curve $ \ga\colon [0,1]\to H $ admits hyperbolic grafting if and
	only if there exist $ t_0,t_1\in [0,1] $ such that 
	$ \Phi(t_1)\Phi(t_0)^{-1} $ has the form
	\begin{equation*}
		\begin{bmatrix}
			r(1-\sin2\theta) & -r\cos2\theta \\
			\cos2\theta & -(1-\sin2\theta) 
		\end{bmatrix}= 
		\begin{bmatrix}
			r(\cos\theta-\sin\theta) & -r(\cos\theta+\sin\theta) \\
			\cos\theta+\sin\theta & \sin\theta-\cos\theta 
		\end{bmatrix}\in \PSL_2(\R)
	\end{equation*}
	for some $ r>0 $ and $ \theta\in (0,\frac{\pi}{2})$. 

	For the proof, note that any two one-parameter groups of hyperbolic
	isometries are conjugate, hence $ \Ga $ may be taken as the group of
	positive homotheties centered at $ 0 $.  By homogeneity, it may also be
	assumed that $ \ta(t_0) = i\in \C $ based at $ i\in H $. Then $ -\ta(t_1) =
	\Im(z)z $ if it is based at $ z \in H $ and tangent to an orbit of $ \Ga $.
	The M\"obius transformations $ M\in \PSL_2(\R) $ satisfying $
	dM_i(i)=-\Im(z)z $ ($ z=re^{2i\theta}\in H $) admit the stated description. 
\end{rmk}


\section{The case where $ (\ka_1,\ka_2) $ is disjoint from $ [-1,1] $}
\label{S:disjoint}

\begin{lem}\label{L:convex}
	In the disk and half-plane models, a curve whose (hyperbolic) curvature is
	everywhere greater than $ 1 $ is locally convex from the Euclidean
	viewpoint, i.e., its total turning is strictly increasing. 
\end{lem}
\begin{proof}
	It suffices to prove this for curves of constant curvature greater than 1,
	because a general curve satisfying the hypothesis is osculated by curves of
	this type.  In $ D $ or $ H $, such curves are represented as Euclidean
	circles traversed in the counterclockwise direction, hence they are locally
	convex.
\end{proof}

\begin{thm}\label{T:disjoint}
	If $ (\ka_1,\ka_2) $ is disjoint from $ [-1,1] $, then each canonical
	subspace of $ \sr C_{\ka_1}^{\ka_2}(u,v) $ is either empty or contractible.
\end{thm}

\begin{proof}
	We will work in the half-plane model $ H $ throughout the proof. Points and
	tangent vectors will thus be regarded as elements of $ \C $.  By
	\pref{P:reduction}, it may be assumed that $ \ka_1 \geq 1 $ and that $ u $
	is parallel to $ 1 \in \Ss^1 $. Let $ \tau \in \R $ be a fixed valid total
	turning for curves in $ \sr C_{\ka_1 }^{\ka_2 }(u,v) $, meaning that $
	e^{i\tau} $ is parallel to $ v \in \C $. 

	Given $ \ga \in \sr C_{\ka_1 }^{\ka_2 }(u,v;\tau) $, let $ \theta_{\ga}
	\colon [0,1] \to [0,\tau] $ be the unique continuous function such that $
	\theta_\ga(0) = 0 $ and
	\begin{equation*}
		\frac{\ta_{\ga}(t)}{\vert \ta_\ga(t)\vert} = e^{i\theta_{\ga}(t)} \text{
		for all $ t\in [0,1] $,}
	\end{equation*}
	where $ \abs{\ } $ denotes the usual absolute value of complex numbers.
	By \lref{L:convex}, $ \theta_\ga $ is a diffeomorphism of $ [0,1] $ onto $
	[0,\tau] $. Thus it may be used as a parameter for $ \ga $.

	Now fix $ \ga_0 \in \sr C_{\ka_1 }^{\ka_2 }(u,v;\tau) $ and set $ \ga_1:=\ga
	$, both viewed as curves in $ H $ and parametrized by the argument $ \theta
	\in [0,\tau] $, as above. Define 
	\begin{equation*}
		\ga_s(\theta) = (1-s)\ga_0(\theta) + s\ga_1(\theta) \in \C \quad (s \in
		[0,1],~\theta \in [0,\tau]).
	\end{equation*}
	Then each $ \ga_s \colon [0,\tau] \to H $ is a smooth curve satisfying $
	\ta_{\ga_s}(0) = u $ and $ \ta_{\ga_s}(1) = v $. Moreover, it has the
	correct total turning $ \tau $, since $ \ta_{\ga_s}(\theta) $ is always
	parallel to $ e^{i\theta} $.  By \rref{R:constant}\?(g) and the definition
	of ``osculation'', the constant-curvature curves osculating $ \ga $ from the
	Euclidean and hyperbolic viewpoints at any $ \theta $ agree. Since $ \ka_1
	\geq 1$,  they are both equal to a certain circle completely contained in $
	H $, traversed counterclockwise; compare \rref{R:orientation}\?(b).
	Therefore the osculating constant-curvature curve to $ \ga_s $ at $ \theta $
	is another circle $ C_s $, the corresponding convex combination of the
	osculating circles $ C_0 $ to $ \ga_0 $ and $ C_1 $ to $ \ga_1 $ at $ \theta
	$ (see \fref{F:circles}). Let $ y_s $ denote the $ y $-coordinate of the
	Euclidean center of $ C_s $ and $ r_s $ its Euclidean radius ($ s \in [0,1]
	$). More explicitly, 
	\begin{equation*}
		y_s = (1-s)y_0+sy_1 \quad \text{and} \quad r_s = (1-s)r_0 + sr_1.
	\end{equation*}
	The \tsl{hyperbolic} diameter $ d_s $ of $ C_s $ is given by
	\begin{equation*}
		d_s = \log \bigg( \frac{y_s+r_s}{y_s-r_s} \bigg).
	\end{equation*}
	Note that $ y_s>r_s $, as this is true for $ s=0,\,1 $. Because 
	\begin{equation*}
		\phantom{\quad (i=0,1)}2\arccoth(\ka_2) < d_i < 2\arccoth(\ka_1) 
		\quad (i=0,1)
	\end{equation*}
	by hypothesis, a trivial computation shows that $ d_s $ satisfies the same
	inequalities for all $ s\in [0,1] $. Hence the curvature of $ \ga_s $ does
	indeed take values inside $ (\ka_1,\ka_2) $, and $ (\ga,s) \mapsto \ga_s $
	defines a contraction of $ \sr C_{\ka_1 }^{\ka_2 }(u,v;\tau) $.
\end{proof}

\begin{figure}[ht]
	\begin{center}
		\includegraphics[scale=.20]{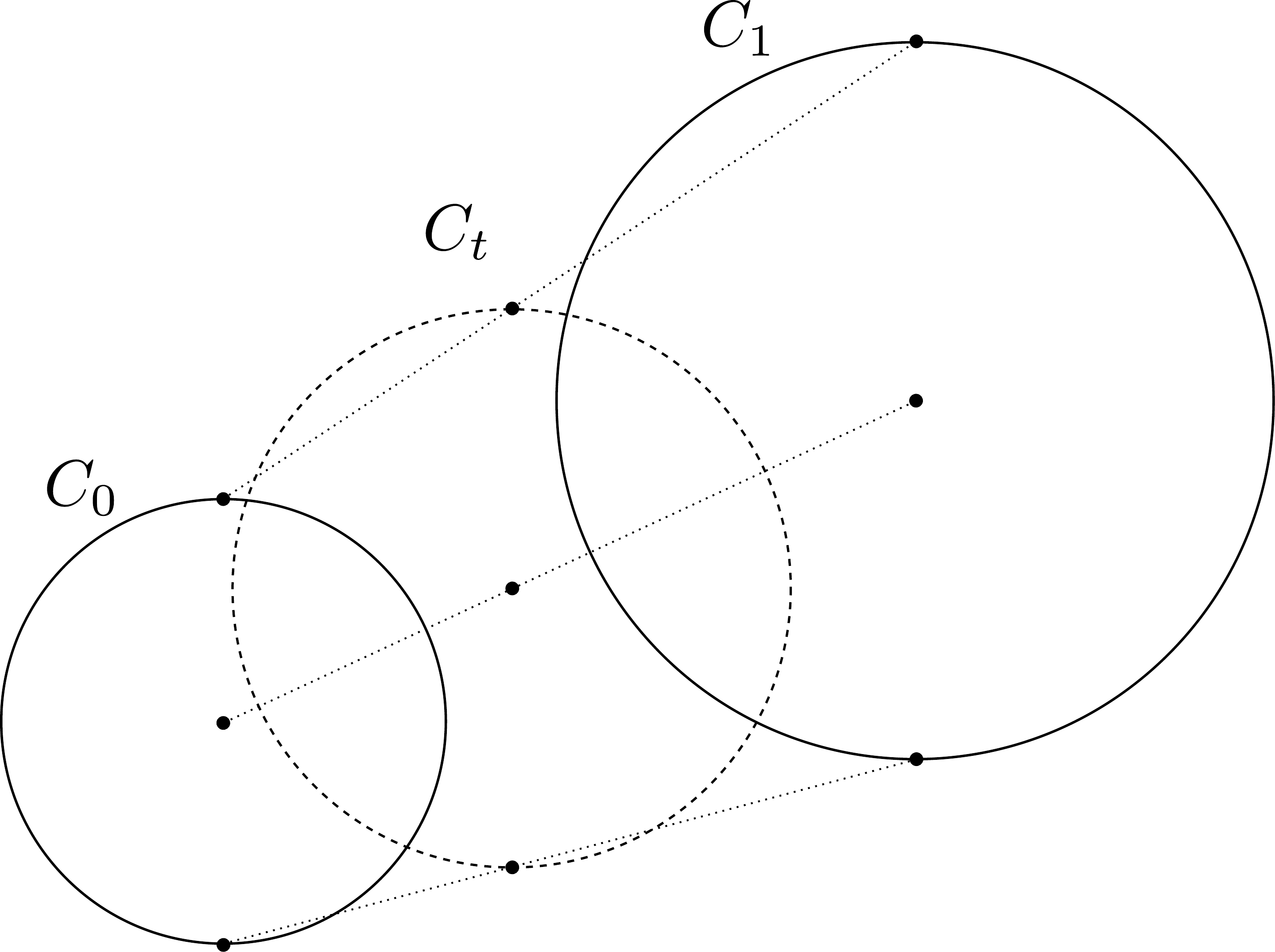}
		\caption{Proof of \tref{T:disjoint}}
		\label{F:circles}
	\end{center}
\end{figure}

\begin{crl}\label{C:disjoint}
	If $ (\ka_1,\ka_2) $ is disjoint from $ [-1,1] $, then $ \sr
	C_{\ka_1}^{\ka_2}(u,v) $ is homeomorphic to the disjoint union of
	countably many copies of the separable Hilbert space.
\end{crl}
\begin{proof}
	This is an immediate consequence of \pref{P:unattainable}\?(c),
	\tref{T:disjoint} and \lref{L:Hilbert}\?(b).
\end{proof}

\begin{urmk}\label{R:disk}
	It is interesting to note that the proof of \tref{T:disjoint} does not work
	in the disk model. To understand what could go wrong, consider the situation
	where $ C_0 $ and $ C_1 $ have the same radius and the midpoint of the
	segment joining their centers is the origin of $ D $ (all concepts here
	being Euclidean).  Then the hyperbolic radius of $ C_{\frac{1}{2}} $ can be
	arbitrarily small compared to that of $ C_0 $ and $ C_1 $.
\end{urmk}

\begin{urmk}\label{R:disjoint}
	The argument in the proof of \tref{T:disjoint} goes through without
	modifications to show that each component of $ \bar{\sr C}_{\ka_1 }^{\ka_2
	}(u,v) $ is contractible if $ [\ka_1,\ka_2] $ is disjoint from $ [-1,1] $.
\end{urmk}


\section{The case where $ (\ka_1,\ka_2) \subs [-1,1] $}\label{S:contained}

In this section we will work with the spaces $ \sr L_{\ka_1 }^{\ka_2
}(u,v) $ which are introduced in \S\ref{S:discontinuous}. These are larger than
$ \sr C_{\ka_1 }^{\ka_2 }(u,v) $ in that they include regular piecewise $ C^2 $
curves. Our proof of \tref{T:contained} uses such curves and is therefore more
natural in the former class. We recommend that the reader ignore the technical
details for the moment and postpone a careful reading of
\S\ref{S:discontinuous}. There is also a way to carry out the proof below in the
setting of $ \sr C_{\ka_1 }^{\ka_2 }(u,v) $: Whenever a discontinuity of the
curvature arises, take an approximation by a smooth curve (which needs to be
constructed); this path is more elementary but also cumbersome. The reader will
notice that similar discussions appear in \cite{Saldanha3}, \cite{SalZueh} and
\cite{SalZueh1}.

Recall the Mercator model $ M $ defined in \dref{D:Mercator}. Suppose that $ \ga
$ is a smooth regular curve in $ M $ which can be written in the form
\begin{equation}\label{E:gammax}
	\ga\colon [a,b] \to M,\quad \ga(x) = (x,y(x))
\end{equation}
for some $ [a,b] \subs (0,\pi) $; in other words, the image of $ \ga $ is the
graph of a function $ y(x) $.  A straightforward computation using
\rref{R:compM} shows that the curvature of $ \ga $ is then given by
\begin{equation}\label{E:curvatureM}
	\ka_\ga = \frac{1}{\sqrt{1+\dot y^2}}\bigg( 
	\frac{\ddot y \sin x}{1+\dot y^2} - \dot y \cos x  \bigg).
\end{equation}
More important than this expression itself is the observation that it does not
involve $ y $, only its derivatives (because vertical translations are
isometries of $ M$). This can be exploited to express geometric properties of $
\ga $ solely in terms of the function $ f = \dot y $, and in particular to
prove the following.\footnote{A similar construction was used in \S3 of 
	\cite{SalZueh1}.}

\begin{thm}\label{T:contained}
	If $ (\ka_1,\ka_2) $ is contained in $ [-1,1] $, then $ \sr C_{\ka_1
	}^{\ka_2 }(u,v) $ is either empty or contractible, hence homeomorphic to
	the separable Hilbert space. 
\end{thm}
\begin{proof}
	By \pref{P:reduction}, we can assume that $ u $ is represented in $ M $ as
	the vector $ 1 \in \Ss^1 $ based at $ (\frac{\pi}{2},0) $. By \lref{L:C^2},
	it suffices to prove that the Banach manifold $ \sr L_{\ka_1 }^{\ka_2 }(u,v)
	$ is either empty or weakly contractible. Let 
	\begin{equation*}
		\Ss^k \to \sr L_{\ka_1}^{\ka_2}(u,v) ,\quad  p\mapsto \ga_p
	\end{equation*}
	be a continuous map. By \lref{L:smoothie}, it can be assumed that each $
	\ga_p $ is smooth. In particular, it is possible to choose $ \bar\ka_2<\ka_2
	$ and $ \bar\ka_1>\ka_1 $ such that the curvatures of all the $ \ga_p $ take
	values inside $ (\bar \ka_1,\bar\ka_2) $. By \lref{L:graph}, any such curve
	$ \ga $ may be parametrized as in \eqref{E:gammax}.	The function $ f = \dot
	y $ satisfies the following conditions, whose meaning will be explained
	below:
	\begin{enumerate} 
		\item [(i)] $ f(a)= \al $ and $ f(b)= \be $.  
		\item [(ii)] $ \int_a^b f(t)\,dt = A_0 $.  
		\item [(iii)] $ \psi_{\bar \ka_1}(x,f(x)) \leq \dot f(x) \leq
			\psi_{\bar\ka_2}(x,f(x)) $ for a.e.~$ x \in [a,b] $,
			where $ \psi_\ka \colon [a,b] \times \R \to \R $ is given by: 
			\begin{equation*}
				\psi_{\ka}(x,z) = \frac{1+z^2}{\sin x} \big(z\cos x + \ka
				\sqrt{1+z^2}\big)\qquad (\ka \in [\bar \ka_1,\bar \ka_2],~
				x\in [a,b],~z\in \R).
			\end{equation*} 

	\end{enumerate} 
	In the present case, $ a = \frac{\pi}{2} $ is the $ x $-coordinate of $
	\ga(a) = \big(\frac{\pi}{2},0\big) $ and $ b $ is the $ x $-coordinate of $
	\ga(b) \in M \subs \C $. The real numbers $ \al =0 $ and $ \be $  in (i) are
	the slopes of $ u $ and $ v $ regarded as vectors in $ \C $. Condition (ii)
	prescribes the $ y $-coordinate of $ \ga(b) $.  Finally, the inequalities in
	(iii) express the fact that the curvature of $ \ga $ takes values in $
	[\bar\ka_1,\bar\ka_2] $; compare \eqref{E:curvatureM}.  Conversely, suppose
	that an absolutely continuous function $ f\colon [a,b] \to \R $ satisfies
	(i)--(iii). If we set 
	\begin{equation*}
		y(x):=\int_a^xf(t)\,dt 
	\end{equation*} 
	and define $ \ga $ through \eqref{E:gammax}, then $ \ga \in \sr
	L_{\ka_1}^{\ka_2}(u,v) $.  Thus one can produce a contraction of the
	original family of curves by constructing a homotopy $ (s,f) \mapsto f_s $
	of the corresponding family of functions $ f=f_1 $, subject to the stated
	conditions throughout, with $ f_0 $ independent of $ f $. This is what we
	shall now do.

	For each $ \ka \in [\bar \ka_1,\bar \ka_2] $, let $ g_\ka $ be the
	solution of the initial value problem
	\begin{equation}\label{E:gka}
		\dot g(x) = \psi_{\ka}(x,g(x)), \quad g(a) = \al,
	\end{equation}
	where $ \al $ and $ \psi_\ka $ are as in conditions (i) and (iii).
	Geometrically, the graph of $ g $ is an arc of the hypercircle of curvature $
	\ka $ with initial unit tangent vector $ u $, represented in $ M $. In
	particular, $ g_\ka $ is defined over all of $ [a,b] $ by \lref{L:graph}. 
	Similarly, for each $ \ka \in [\bar \ka_1,\bar \ka_2] $, let 
	$ h_{\ka} $ be the solution of the initial value problem 
	\begin{equation*}
		\dot h(x) = \psi_{\ka}(x,h(x)), \quad h(b) = \be.
	\end{equation*} 
	Geometrically, the graph of $ h_\ka $ is an arc of the hypercircle of
	curvature $ \ka $ whose final unit tangent vector is $ v $. Although $
	h_{\ka} $ is smooth, it is possible that it is not defined over all of $
	[a,b] $; if its maximal domain of definition is $ (a',b] $ for some $ a' > a
	$, then we extend $ h_\ka $ to a function $ [a,b] \to \R \cup \se{\pm
	\infty} $ by setting it equal to $ \lim_{x \to a'{+}} h_{\ka}(x) $ over $
	[a,a'] $.

	Using this geometric interpretation, it follows from \lref{L:unattainable}
	that 
	\begin{equation*}
		g_{\bar\ka_1},\,h_{\bar\ka_1} \leq f \leq g_{\bar\ka_2},\,h_{\bar\ka_2}.
	\end{equation*}
	Moreover, from the same result one deduces that 
	\begin{equation}\label{E:increasing}
		g_{\ka}(x) < g_{\ka'}(x) \ \ \text{for all\ \ $ x > a $\ \ if\ \ 
		$ \ka < \ka' $}.
	\end{equation}
	For each $ \la,\,\mu \in [\bar \ka_1,\bar \ka_2]$, define
	$ f_{(\la,\mu)} \colon [a,b]\to \R$ by
	\begin{equation}\label{E:median}
		f_{(\la,\mu)}(x)=\midd \big(h_{\bar \ka_1}(x) \,,\, g_{\la}(x) \,,\,
		f(x) \,,\, g_{\mu}(x) \,,\, h_{\bar\ka_2}(x) \big).
	\end{equation}
	Notice first that this function does not take on infinite values, 
	since $ f,\,g_\la,\,g_\mu $ are real functions. Similarly, since three of
	the functions above (namely, $ f,\,g_{\la},\,g_{\mu} $) take the value $ \al
	$ at $ a $, and three of them  (namely, $ f,\,h_{\bar\ka_1},\,h_{\bar\ka_2}
	$) take the value $ \be $ at $ b $, $ f_{(\la,\mu)} $ automatically
	satisfies condition (i). It is easy to verify that 
	it is Lipschitz and satisfies (iii) as well (see \lref{L:median} below). 

	It remains to choose $ (\la,\mu) $ appropriately to ensure that it satisfies
	condition (ii).  Let 
	\begin{alignat*}{9}
		\ka_+ = \min &\set{\ka \in [\bar \ka_1,\bar\ka_2]}{f(x) \leq g_{\ka}(x)
		\text{ for all $ x \in [a,b] $}}, \\
		\ka_- = \max &\set{\ka \in [\bar \ka_1,\bar\ka_2]}{g_{\ka}(x) \leq f(x)
		\text{ for all $ x \in [a,b] $}}, \\
		\De=&\set{(\la,\mu)\in [\ka_-,\ka_+]}{\la \leq \mu} \text{\quad
			(cf.~\fref{F:delta})}.
	\end{alignat*} 

	\begin{figure}[ht]
		\begin{center}
			\includegraphics[scale=.43]{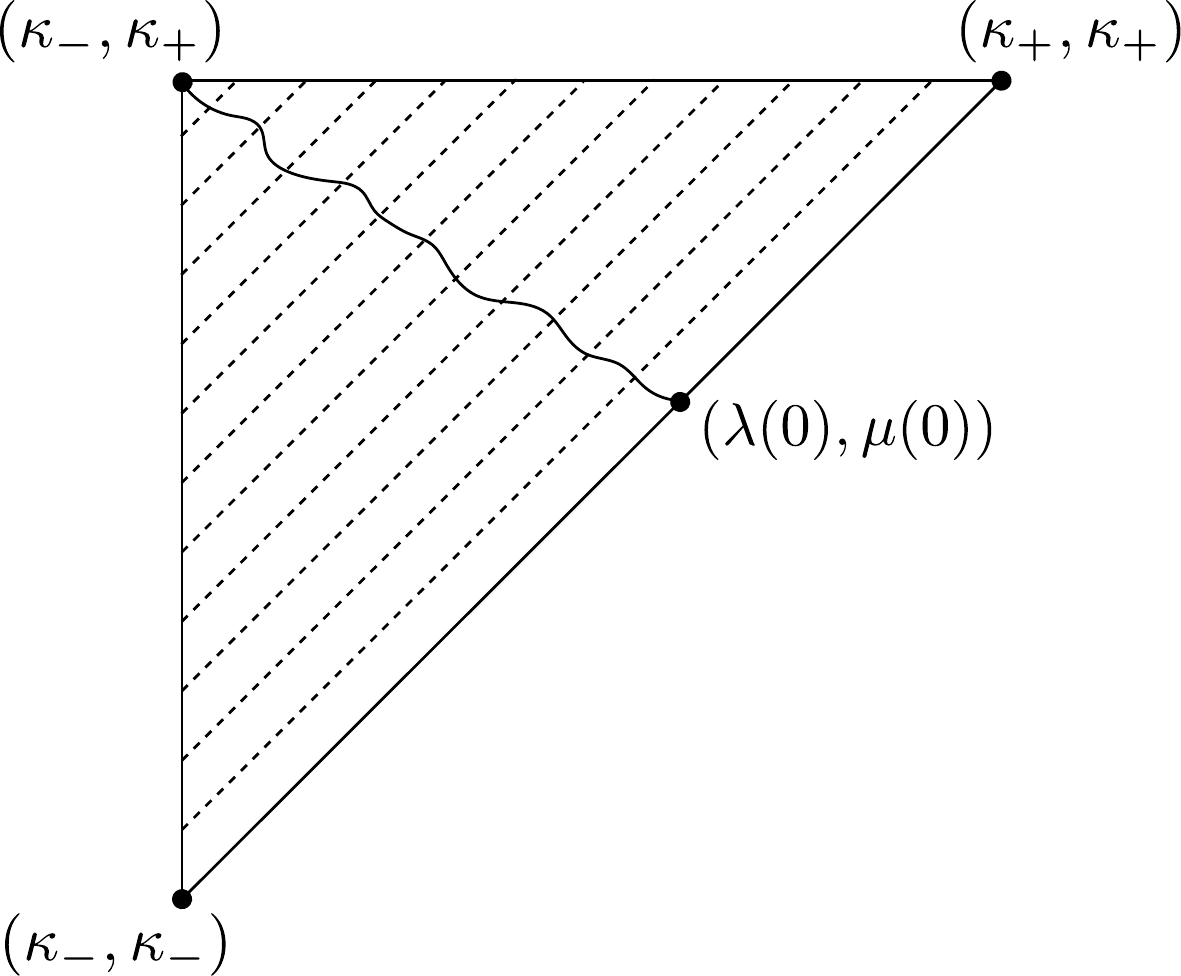}
			\caption{
				A diagram illustrating the triangle $ \De $ in the $ (\la,\mu)
				$-plane. The dashed segments are the intersections of $ \De $
				with the lines $ \set{(\la,\mu) \in \R^2}{\mu-\la =
				s(\ka_+-\ka_-)} $ $ (s\in [0,1]) $.
			} 
			\label{F:delta}
		\end{center}
	\end{figure}

	Define $A \colon \De\to \R$ to be the area under the
	graph of $f_{(\la,\mu)}$:
	\begin{equation*}
		A(\la,\mu)=\int_a^bf_{(\la,\mu)}(x)\?dx.
	\end{equation*}
	Continuous dependence of the solutions of \eqref{E:gka} on $ \ka $ implies
	that $A$ is a continuous function of $ (\la,\mu) $.
	Moreover, from 
	\begin{equation*}
		h_{\bar \ka_1},\,g_{\ka_-} \leq f \leq g_{\ka_+},\,h_{\bar \ka_2}, 
	\end{equation*}
	one deduces that
	\begin{equation*}
		f_{(\ka_-,\mu)} \leq f \leq f_{(\la,\ka_+)}
	\end{equation*}
	for all $ \la,\,\mu \in [\ka_-,\ka_+] $. Consequently, because the integral
	of $ f $ equals $ A_0 $, for each $s\in [0,1]$ the set 
	\begin{equation*}
		\set{(\la,\mu)\in \De}{\mu-\la=s(\ka_+-\ka_-)\text{\ and\ }
	A(\la,\mu) = A_0}
	\end{equation*} 
	is nonempty (see \fref{F:delta}). In fact, it consists of a unique point $
	(\la(s),\mu(s)) $. To establish this, it suffices to show that $ A(\la,\mu)
	$ is a strictly increasing function of both $ \la $ and $ \mu $. 
	Now if 
	\begin{equation*}
	\ka_- \leq \la < \la' \leq \mu \leq \ka_+,
	\end{equation*}
	then the set of all $ x \in [a,b] $ for which $ g_{\la}(x) < f(x) <
	g_{\la'}(x) $ is nonempty, by \eqref{E:increasing} and the choice of $
	\ka_{\pm} $. Therefore $ f_{(\la,\mu)} < f_{(\la',\mu)} $ holds over a set
	of positive measure, while the nonstrict inequality holds everywhere by
	\eqref{E:increasing}. This proves strict monotonicity with respect to $ \la
	$; the argument for $ \mu $ is analogous.  Continuity of $ A $ implies
	continuity of the curve $ s\mapsto (\la(s),\mu(s)) \in \De $ (which is
	depicted in bold in \fref{F:delta}).  
	The functions
	\begin{equation*}
	f_s\colon [a,b]\to \R,\quad f_s=f_{(\la(s),\mu(s))}
	\end{equation*}	
	satisfy all of conditions (i)--(iii) by construction, and they depend
	continuously on $ f $ and $ s $.  Let $ \ka_0 = \la(0) = \mu(0) $; then 
	\begin{equation*}
	f_0 = \midd 
	\big( h_{\bar\ka_1} \,,\, g_{\ka_0} \,,\, h_{\bar\ka_2} \big).
	\end{equation*}
	By \eqref{E:increasing}, there is at most one value of $ \ka \in [\bar
	\ka_1,\bar \ka_2] $ for which the integral of $ \midd ( h_{\bar\ka_1} \,,\,
	g_{\ka} \,,\, h_{\bar\ka_2} ) $ equals $ A_0 $.  This implies that $ \ka_0
	$, and hence $ f_0 $, is independent of $ f $.  Therefore $ (f,s) \mapsto
	f_s $ is indeed a contraction.
\end{proof}

The following fact was used without proof above.

\begin{lem}\label{L:median}
	The function $ f_{(\la,\mu)} $ of \eqref{E:median} is Lipschitz 
	and satisfies (iii).
\end{lem}
\begin{proof}
	More generally, let $ \phi $ be the median of $ \phi_1,\dots,\phi_{2n+1}
	\colon [a,b] \to \R$. If each $ \phi_k $ is $ c $-Lipschitz, then $ \phi $
	is $ c $-Lipschitz. Hence $ \phi $ is absolutely continuous and its
	derivative exists a.e..  Furthermore, if $ \phi $ and each $ \phi_k $ are
	differentiable at $ x $, then $ \phi'(x) = \phi_k'(x) $ for some $ k $ such
	that $ \phi(x) = \phi_k(x) $. In particular, if each $ \phi_k $ satisfies
	the inequalities in (iii) (with $ \phi_k $ in place of $ f $), then so does
	$ \phi $.

	The function $ f_{(\la,\mu)} $ does not immediately conform to this
	situation because $ h_{\bar \ka_i} $ $ (i=1,2) $ may take on infinite
	values. This can be circumvented by subdividing $ [a,b] $ into
	at most three subintervals (where none, one or both of $ h_{\bar \ka_i} $
	are infinite) and applying the preceding remarks.
\end{proof}


\section{Spaces of curves with constrained curvature of class $ C^r $}
\label{S:general}

In this section we consider spaces of curves with constrained curvature on an
arbitrary surface, not necessarily hyperbolic nor orientable.  We study their
behavior under covering maps and show that they are always nonempty if $
S $ is compact; this should be contrasted with \lref{L:attainable}. 

A \tdfn{surface} is a smooth Riemannian 2-manifold. Given a regular curve $ \ga
\colon [0,1] \to S $,
its \tdfn{unit tangent} $\ta=\ta_\ga$ is the lift of $ \ga $ to
the unit tangent bundle $UTS$ of $S$: 
\begin{equation*}
	\ta\colon [0,1]\to UTS, \quad \ta(t)=\frac{\dot\ga(t)}{\abs{\dot\ga(t)}}.
\end{equation*}
Now let an orientation of $ TS_{\ga(0)} $ be fixed.  The \tdfn{unit
normal} to $\ga$ is the map $\no=\no_\ga\colon [0,1]\to UTS$ determined by the
requirement that $ (\ta(t),\no(t)) $ is an orthonormal basis of $ TS_{\ga(t)} $
whose parallel translation to $ \ga(0) $ (along the inverse of $ \ga $) is
positively oriented, for each $ t \in [0,1] $. Assuming $\ga$ is twice
differentiable, its \tdfn{curvature} $\ka=\ka_\ga$ is given by
\begin{equation}\label{E:curvature}
	\ka:=\frac{1}{\abs{\dot\ga}}\Big\langle\frac{D\ta}{dt},\no\Big\rangle
	=\frac{1}{\abs{\dot\ga}^2}\Big\langle\frac{D\dot\ga}{dt},\no\Big\rangle.
\end{equation} 
Here $D$ denotes covariant differentiation (along $\ga$).

\begin{urmk}\label{R:alternative}
	If $ S $ is nonorientable, it is more common to define the 
	(unsigned) curvature of $ \ga \colon [0,1] \to S $ by
	\begin{equation*}
		\ka = \frac{1}{\abs{\dot \ga}}\abs{\frac{D\ta}{dt}}.
	\end{equation*}
	If $ S $ is orientable, the usual definition coincides with
	\eqref{E:curvature}, but $ \no $ is defined by the condition that $
	(\ta(t),\no(t)) $ be positively oriented with respect to a specified
	orientation of $ S $, rather than the parallel translation of an orientation
	of $ TS_{\ga(0)} $.  These two definitions are equivalent, since an
	orientation of $ TS_{\ga(0)} $ determines an orientation of $ S $ if the
	latter is orientable. The definition that we have chosen has the advantage
	of allowing arbitrary bounds for the curvature of a curve on a nonorientable
	surface, and in particular the concise formulation of \lref{L:covering}
	below. 
\end{urmk}

A geometric interpretation for the curvature is the following: Let $
v\colon [0,1]\to UTS $ be any smooth parallel vector field along $ \ga $, and
let $ \theta\colon [0,1]\to \R $ be a function measuring the oriented angle from
$ v(t) $ to $ \ta(t) $.  Then a trivial computation shows that $ \dot\theta =
\ka \abs{\dot\ga} $. In particular, the \tdfn{total curvature}
\begin{equation*}
	\int_0^1\ka(t)\abs{\dot\ga(t)}\,dt
\end{equation*}
of $ \ga $ equals $ \theta(1)-\theta(0) $. 

In all that follows, the curvature bounds $ \ka_1<\ka_2 $ are allowed to take
values in $ \R\cup\se{\pm\infty} $, $ S $ denotes a surface and $ u,\,v\in UTS
$. Moreover, it is assumed that an orientation of $ TS_p $, where $ p $ is the
basepoint of $ u $, has been fixed.

\begin{dfn}[spaces of $ C^r $ curves]\label{D:Curve} 
	Define $\sr CS_{\ka_1}^{\ka_2}(u,v)^r$ to be the set, endowed with the $C^r$
	topology (for some $ r \geq 2 $), of all $C^r$ regular curves $\ga\colon
	[0,1]\to S$ such that: 
	\begin{enumerate} 
		\item [(i)] $\ta_\ga(0)=u$ and $\ta_\ga(1)=v$;
		\item [(ii)] $\ka_1<\ka_\ga(t)<\ka_2$ for each $t\in [0,1]$.  
	\end{enumerate}
\end{dfn} 

\begin{urmk}
	Of course, whether $ S $ is orientable or not, the topological properties
	(or even the voidness) of $ \sr CS_{\ka_1 }^{\ka_2 }(u,v) $ may be sensitive
	to the choice of orientation for $ TS_p $. More precisely, if $ \bar S $
	denotes the same surface $ S $ with the opposite orientation of $
	TS_p $, then $ \sr C\bar S_{\ka_1 }^{\ka_2 }(u,v) = \sr
	CS_{-\ka_2}^{-\ka_1}(u,v) $. 
\end{urmk}
	It will follow from (\ref{L:C^2}) that $r$ is irrelevant in the
	sense that different values yield spaces which are homeomorphic.
	Because of this, $ \sr CS_{\ka_1}^{\ka_2}(u,v)^r $ is denoted simply by $
	\sr CS_{\ka_1}^{\ka_2}(u,v) $ in this section.  

\begin{lem}\label{L:contractible}
	Define $ \sr CS_{\ka_1}^{\ka_2}(u,\cdot) $ as in \dref{D:Curve}, except that
	no condition is imposed on $ \ta_\ga(1) $, and similarly for $ \sr
	CS_{\ka_1}^{\ka_2}(\cdot,v) $ and $ \sr CS_{\ka_1}^{\ka_2}(\cdot,\cdot) $.
	Then:
	\begin{enumerate}
		\item [(a)] $ \sr CS_{\ka_1}^{\ka_2}(u,\cdot) $ and $ \sr
			CS_{\ka_1}^{\ka_2}(\cdot,v) $ are contractible.
		\item [(b)] $ \sr CS_{\ka_1}^{\ka_2}(\cdot,\cdot) $ is homotopy
			equivalent to $ UTS $. 
	\end{enumerate}
\end{lem}
\begin{proof}
	By \lref{L:Hilbert}\?(b), to prove that $ \sr CS_{\ka_1 }^{\ka_2 }(u,\cdot)
	$ is contractible, it is actually sufficient to show that it is \tsl{weakly}
	contractible. Let 
	\begin{equation*}
		K \to \sr CS_{\ka_1}^{\ka_2}(u,\cdot),\quad p \mapsto \ga^p, 
	\end{equation*}
	be a continuous map, where $ K $ is a compact space. By a preliminary
	homotopy, each $ \ga^p $ may be reparametrized with 
	constant speed. Let $ \la(p) = \text{length}(\ga^p) $ and $ 0 <
	\la < \inf_{p \in K}\la(p) $.  The curves can be shrunk
	through the homotopy
	\begin{equation*}
		(s,p)\mapsto \ga_s^p,\quad 
		\ga^p_s(t):=\ga^p\Big(\la(p)^{-1}\big[(1-s)\la+s\la(p)\big]\?t\Big)
		\quad (s,\,t\in [0,1],~p \in K) 
	\end{equation*}
	so that all $ \ga_0^p $ have length $ \la $, which can be chosen smaller
	than the injectivity radius of $ S $ at the basepoint of $ u $.  Then each $
	\ga_0^p $ is determined solely by its curvature, and conversely any function
	$ \ka \colon [0,1] \to (\ka_1,\ka_2) $ of class $ C^{r-2} $ 
	determines a unique curve of constant speed $ \la $ in $ S $ having $ u $
	for its initial unit tangent vector.  But the set of all such functions is
	convex.  
	
	Reversal of orientation of curves yields a homeomorphism between and $ \sr
	CS_{-\ka_2 }^{-\ka_1 }(-v,\cdot) $ and $ \sr CS_{\ka_1 }^{\ka_2 }(\cdot,v)
	$, hence a space of the latter type is also contractible.
	
	For (b), consider the map $ f\colon UTS \to \sr C_{\ka_1 }^{\ka_2
	}(\cdot,\cdot) $ which associates to $ u $ the unique curve of
	constant curvature $ \frac{1}{2}(\ka_1+\ka_2) $ having $ u $ for its initial unit
	tangent and length equal to half the injectivity radius of $ S $ at the
	basepoint of $ u $, parametrized with constant speed. Using the argument of
	the first paragraph, one deduces that $ f $ is a weak homotopy inverse of
	\begin{equation*}
		g \colon \sr CS_{\ka_1 }^{\ka_2 }(\cdot,\cdot) \to UTS, \quad
		g(\ga) = \ta_\ga(0).
	\end{equation*}
	Therefore $ \sr C_{\ka_1 }^{\ka_2 }(\cdot,\cdot) $ is homotopy equivalent to
	$ UTS $, by \lref{L:Hilbert}\?(b).
\end{proof}

\begin{lem}\label{L:covering}
	Let $ q\colon \te{S} \to S $ be a Riemannian covering (a covering map which
	is also a local isometry) and $ u,\,v\in UTS $. 
	Suppose that $ dq \colon \big(T\te{S}_{\te{p}},\te{u}\big) \to
	\big(TS_p,u\big) $ preserves the chosen orientation of these tangent planes.
	Then $ \te{\ga} \mapsto q\circ \te{\ga} $ yields a homeomorphism
	\begin{equation*}
		\bdu_{\te{v}\in dq^{-1}(v)}\sr C\te{S}_{\ka_1}^{\ka_2}(\te{u},\te{v}) \home 
		\sr CS_{\ka_1}^{\ka_2}(u,v).
	\end{equation*}
\end{lem}

\begin{proof}
	Let $ \ga \in \sr CS_{\ka_1 }^{\ka_2 }(u,v) $ and $ \te{\ga} \colon [0,1]
	\to \te{S} $ be its lift to $ \te{S} $ starting at $ \te{p} $. Since $ dq $
	is an isometry, 
	\begin{equation*}
		dq(\ta_{\te{\ga}}) = \ta_\ga \quad \text{and} \quad 
		dq(\no_{\te{\ga}}) = \pm \no_\ga.
	\end{equation*}
	Moreover,  $ dq(\no_{\te{\ga}}(0)) = \no_{\ga}(0) $ by the hypothesis
	regarding orientations, hence $ dq(\no_{\te{\ga}}) = \no_\ga $ by
	continuity.  Now by \eqref{E:curvature},  $ \ka_{\te{\ga}} = \ka_\ga $.
	Therefore $ \te{\ga} \in \sr C\te{S}_{\ka_1 }^{\ka_2 }(\te{u},\te{v}) $ for
	some lift $ \te{v} $ of $ v $.  Conversely, if $ \te{\ga} \in \sr
	C\te{S}_{\ka_1 }^{\ka_2 }(\te{u},\te{v}) $, then $ \ga = q\circ \te{\ga} \in
	\sr CS_{\ka_1 }^{\ka_2 }(u,v) $ by the same reasons.  Since projection and
	lift (starting at $ \te{p} $) are inverse operations, the asserted
	homeomorphism holds.
\end{proof}

This result is especially useful when $ S $ is a space form (e.g., a hyperbolic
surface); for then $ S $ is the quotient by a discrete group of isometries of
the model simply-connected space of the same curvature, which is much more
familiar. The lemma also furnishes a reduction to the orientable case by taking
$ \te{S} $ to be the two-sheeted orientation covering of $ S $.

\begin{exr}\label{X:symmetric}
	Suppose that $ (\ka_1,\ka_2) $ is symmetric about 0. Then the conclusion of 
	\lref{L:covering} holds regardless of whether $ dq $ preserves orientation at $
	\te{p} $. (\tit{Hint:} See the remark following \dref{D:Curve}.)
\end{exr}

\begin{urmk}
	If closed or half-open intervals were used instead in \dref{D:Curve}, then
	substantial differences would only arise in marginal cases. For instance, in
	the situation of \lref{L:unattainable}, if $ v $ is tangent to $ \bd R $,
	then $ \sr C_{\ka_1 }^{\ka_2 }(u,v) $ is empty, while $ \bar{\sr C}_{\ka_1
	}^{\ka_2 }(u,v) $ may not be. The original definition is more convenient to
	work with since the resulting spaces are Banach manifolds;  compare Example
	1.1 in \cite{SalZueh1}. 
\end{urmk}

\begin{thm}\label{T:compact} 
	Let $S$ be a compact connected surface.  Then 
	$ \sr CS_{\kappa_1}^{\kappa_2}(u,v)\neq \emptyset $ for any choice of
	$\ka_1<\ka_2$ and $u,\,v\in UTS$.
\end{thm} 
\begin{proof} 
	By passing to the orientation covering if necessary, it may be assumed that
	$ S $ is oriented.  Let $\ka_1<\ka_2$ be fixed. For $u,\, v\in UTS$, write
	$u\prec v$ if $\sr CS_{\kappa_1}^{\kappa_2}(u,v)\neq \emptyset$. Notice that
	$\prec$ is transitive: Given curves in $ \sr CS_{\ka_1 }^{\ka_2 }(u,v) $ and $
	\sr CS_{\ka_1 }^{\ka_2 }(v,w)  $, their concatenation starts at
	$ u $ and ends at $ w $; its curvature may fail to exist at the point of
	concatenation, but this can be fixed by taking a smooth approximation.
	Let 
	\begin{equation*}
		F_u=\set{v\in UTS}{u \prec v}, \quad 
		E_u=\set{v\in UTS}{u \prec v \textrm{ and } v \prec u}.
	\end{equation*} 
	It is clear from the definition that $F_u\neq \emptyset$ for
	all $u$, and that $F_u$ and $E_u$ are open subsets of $ UTS $
	(cf.~\lref{L:submersion}). Moreover, the
	family $(F_u)_{u \in UTS}$ covers $UTS$.  Indeed, given $v$, we can  find
	$u$ such that $\sr CS_{-\ka_2}^{-\ka_1}(-v,-u)\neq \emptyset$. Reversing the
	orientation of a curve in the latter set, we establish that $\sr
	CS_{\kappa_1}^{\kappa_2}(u,v) \neq \emptyset$, that is, $v\in F_u$. 

	Since $UTS$ is compact, it can be covered by finitely many of the $F_u$. Let
	$UTS=F_{u_1}\cup \dots \cup F_{u_m}$ be a minimal cover. We claim that
	$m=1$.

	Assume that $m > 1$. If $u_i \prec u_j$ then $F_{u_i} \sups F_{u_j}$, and
	therefore by minimality $i = j$. Since $u_i \in UTS$ and $u_i \notin
	F_{u_j}$ for $j \ne i$,  we deduce that $u_i \in F_{u_i}$ for each $i$. The
	open sets $E_{u_i}$ are thus nonempty and disjoint. On the other hand,
	every $F_{u_i}$ must intersect some $F_{u_j}$ with $j\neq i$, as $UTS$ is
	connected. Choose $i\neq j$ such that $F_{u_i} \cap F_{u_j}\neq \emptyset$.
	It is easy to see that $F_{u_i} \cap F_{u_j}$ must be disjoint from
	$E_{u_i}$. Thus, if $V$ is the interior of $F_{u_i} \ssm E_{u_i}$, then
	$V\neq \emptyset$. We will obtain a contradiction from this. By definition,
	there exist $u_\ast \in E_{u_i}$, $v_\ast \in V$ and $\gamma_\ast \in \sr
	CS_{\kappa_1}^{\kappa_2}(u_\ast,v_\ast)$.  Let $\ga_\ast\colon [0,L]\to S$ be
	parametrized by arc-length and $\ka_\ast\colon [0,L]\to (\ka_1,\ka_2)$
	denote its curvature.

	The tangent bundle $ TS $ has a natural volume form coming from the
	Riemannian structure of $S$. A theorem of Liouville states that the geodesic
	flow preserves volume in $ TS $. Since $S$ is oriented, given $\theta\in \R$
	we may define a volume-preserving bundle automorphism on $ UTS $ by
	$w\mapsto \cos\theta w+\sin\theta J(w)$ (where $ J $ is ``multiplication by
	$ i $''). Let $Y_0$ and $Z$ be the vector
	fields on $ UTS $ corresponding to the geodesic flow and to counterclockwise
	rotation, respectively. Then for any $\kappa\in \R$, the vector field
	$Y_\kappa = Y_0 + \kappa Z$ defines a volume-preserving flow on $UTS$; the
	projections of its orbits on $S$ are curves parametrized by arc-length of
	constant curvature $\kappa$. 

	By definition, the open set $V$ is forward-invariant under each of these
	flows for $\ka\in (\ka_1,\ka_2)$. Define a map $G\colon UTS\to UTS$ as
	follows: Given $u\in UTS$, $G(u) = \ta_\eta(L)$, where $\eta\colon [0,L]\to
	S$ is the unique curve in $\sr CS_{\kappa_1}^{\kappa_2}(u,\cdot)$,
	parametrized by arc-length, whose curvature is $\ka_\ast$. Then $G$ must be
	volume-preserving, $G(V)\subs V$, but there exists a neighborhood of
	$u_\ast$ contained in $E_{u_i}$ which is taken by $G$ to a neighborhood of
	$v_\ast$ contained in $V$, a contradiction.

	We conclude that $m=1$, so that $ UTS = F_{u_1} $. Furthermore, $F_{u_1}\ssm
	E_{u_1}$ must have empty interior by the preceding argument. Hence
	$E_{u_1}$ is a dense open set in $UTS$. Let $u,\,v\in UTS$ be given. Since
	$F_u$ is open, there exists $ v_1\in F_u \cap  E_{u_1} $. Then $ u\prec
	v_1 $. Since $v_1\prec u_1$ and $u_1\prec v$, we deduce that $u\prec v$; in
	other words, $\sr CS_{\kappa_1}^{\kappa_2}(u,v)\neq \emptyset$.  
\end{proof}

\subsection*{Spaces of closed curves without basepoints}	

Just as in algebraic topology one considers homotopies with and without
basepoint conditions, one can also study the space of all smooth closed curves
on $ S $ with curvature in an interval $ (\ka_1,\ka_2) $ but no
restrictions on the initial and final unit tangents. Let this space be denoted
by $ \sr CS_{\ka_1}^{\ka_2} $. In some regards this class may seem more
fundamental than its basepointed version, the class of spaces $ \sr CS_{\ka_1
}^{\ka_2 }(u,v) $ considered thus far.  However, in the Hirsch-Smale theory of
immersions, such basepoint conditions arise naturally. For instance, thm.~C of
\cite{Smale} states that $ \sr CS_{-\infty}^{+\infty}(u,u) $ is
(weakly) homotopy equivalent to the loop space $ \Om(UTS)(u) $ consisting of all
loops in $ UTS $ based at $ u $. Moreover, even if one is interested only in
closed curves, it is often helpful to study $ \sr CS_{\ka_1 }^{\ka_2 } $
by lifting its elements to the universal cover of $ S $, and these lifts need
not be closed. A further point is provided by the following result.  Recall that
$ UTS $ is diffeomorphic to $ \R^2 \times \Ss^1 $ if $ S = \R^2 $ or $ S = \Hh^2
$, and to $ \SO_3 \home \RP^3 $ if $ S=\Ss^2 $. 

\begin{lem}\label{L:closed}
	Let $ S $ be a simply-connected complete surface of constant curvature and 
	$ u \in UTS $ be arbitrary. Then $ \sr CS_{\ka_1 }^{\ka_2 } $ is
	homeomorphic to $ UTS \times \sr CS_{\ka_1 }^{\ka_2 }(u,u) $. 
\end{lem}
\begin{proof}
	The group of orientation-preserving isometries of such a surface acts simply
	transitively on $ UTS $. Given $ v \in UTS $, let $ g_v $ denote its unique
	element mapping $ v $ to $ u $. Define
	\begin{equation*}
		f \colon \sr CS_{\ka_1 }^{\ka_2 } \to UTS \times \sr C_{\ka_1 }^{\ka_2
		}(u,u),\quad f(\ga) = (\ta(0) , g_{\ta(0)} \circ \ga ).
	\end{equation*}
	This is clearly continuous, and so is its inverse, which is given by $
	(v,\eta) \mapsto g_v^{-1} \circ \eta $.
\end{proof}

Regard the elements of $ \sr CS_{\ka_1 }^{\ka_2 } $ as maps $ \Ss^1 \to S $.
There is a natural projection $ \sr CS_{\ka_1 }^{\ka_2 } \to UTS $ taking a
curve $ \ga\colon \Ss^1 \to S $ to its unit tangent at $ 1 \in \Ss^1 $. This is
a fiber bundle if $ \ka_1=-\infty $ and $ \ka_2=+\infty $ since $ S $ is locally
diffeomorphic to $ \R^2 $ and the group of diffeomorphisms of the latter acts
transitively on $ UT\R^2 $. It may be a fibration in certain other special
cases, as in the situation of \lref{L:closed}, but in general this cannot be
guaranteed. For instance, if $ S=T^2 $ is a flat torus, then the homotopy type
of the fibers $ \sr CS_{-1}^{+1}(u,u) $ of the map $ \sr CS_{-1}^{+1} \to UTS $
is not locally constant; this follows from an example in the introduction of
\cite{SalZueh2}. We believe that little is known about the topology of $ \sr
CS_{\ka_1}^{\ka_2} $ beyond what is implied by \lref{L:closed}.


\section{Spaces of curves with discontinuous curvature}\label{S:discontinuous}

Suppose that $\ga\colon [0,1]\to S$ is a smooth regular curve and, as always, $
TS_{\ga(0)} $ has been oriented. Let $\sig\colon [0,1]\to \R^+$ denote its speed
$ \abs{\dot\ga} $ and $\ka$ its curvature.  Then $\ga$ and $\ta=\ta_\ga\colon
[0,1]\to UTS$ satisfy: 
\begin{equation}\label{E:de}
	\begin{cases} 
	\dot\ga=\sig\ta  \\ 
	\frac{D\ta}{dt}=\sig\ka\? \no
	\end{cases}
	\quad\text{and}\quad \ta(0)=u\in UTS.  
\end{equation}  
Thus, $\ga$ is uniquely determined by $ u $ and the functions $\sig,\,\ka$. One
can define a new class of spaces by relaxing the conditions that $\sig$ and
$\ka$ be smooth.

Let $h\colon (0,+\infty)\to \R$ be the diffeomorphism
\begin{equation*}\label{E:thereason} 
	h(t)=t-t^{-1}.  
\end{equation*} 
For each pair $\ka_1<\ka_2\in \R$, let $h_{\ka_1,\,\ka_2} \colon
(\ka_1,\ka_2)\to \R$ be the diffeomorphism 
\begin{equation*}
	h_{\ka_1,\,\ka_2}(t)=(\ka_1-t)^{-1}+(\ka_2-t)^{-1} 
\end{equation*} 
and, similarly, set 
\begin{alignat*}{10}
	&h_{-\infty,+\infty}\colon \R\to
	\R,\qquad  & &t\mapsto t \\
	&h_{-\infty,\ka_2}\colon
	(-\infty,\ka_2)\to \R,\qquad &  &t\mapsto t+(\ka_2-t)^{-1} \\
	&h_{\ka_1,+\infty}\colon (\ka_1,+\infty)\to \R, \qquad &
	&t\mapsto t+(\ka_1-t)^{-1}. 
\end{alignat*} 
Notice that all of these functions are monotone increasing, hence so are their
inverses.  Moreover, if $\hat{\ka}\in L^2[0,1]$, then
$\ka=h_{\ka_1,\ka_2}^{-1}\circ \hat\ka\in L^2[0,1]$ as well. This is obvious if
$(\ka_1,\ka_2)$ is bounded, and if one of $\ka_1,\ka_2$ is infinite, then it is
a consequence of the fact that $h_{\ka_1,\ka_2}^{-1}(t)$ diverges linearly to
$\pm \infty$ with respect to $t$.  In what follows, $\L$ denotes the separable
Hilbert space $L^2[0,1]\times L^2[0,1]$.  
\begin{dfn}[admissible curve]
	A curve $\ga\colon [0,1]\to S$ is $(\ka_1,\ka_2)$-\tdfn{admissible} if there
	exists $(\hat \sig,\hat \ka)\in \L$ such that $\ga$ satisfies \eqref{E:de}
	with
	\begin{equation}\label{E:Sobolev} 
		\sig=h^{-1}\circ \hat{\sig}\text{\quad and\quad }
		\ka=\?h^{-1}_{\ka_1,\,\ka_2}\circ \hat{\ka}.  
	\end{equation}
	When it is not important to keep track of the bounds $\ka_1,\ka_2$, we will
	simply say that $\ga$ is \tdfn{admissible}.  
\end{dfn}

The system \eqref{E:de} has a unique solution for any $(\hat \sig,\hat \ka)\in
\L$ and $u\in UTS$. To see this, we use coordinate charts for $TS$ derived from
charts for $S$ and apply thm.~C.3 on p.~386 of \cite{Younes} to the resulting
differential equation, noticing that $S$ is smooth and $\sig,\, \ka\in
L^2[0,1]\subs L^1[0,1]$.  Furthermore, if we assume that $ S $ is complete, then
the solution is defined over all of $[0,1]$. The resulting maps $\ga\colon
[0,1]\to S$ and $\ta\colon [0,1]\to TS$ are absolutely continuous (see p.~385 of
\cite{Younes}), and so is $ \no $. Using that $\gen{\ta,\no}\equiv 0$
and differentiating, we obtain, in addition to \eqref{E:de}, that 
\begin{equation*}
	\frac{D\no}{dt}=-\sig \ka\?\ta \text{\quad and \quad} 
	\abs{\ta(t)}=\abs{\no(t)}=\abs{u}=1\ \
	\text{for all $t\in [0,1]$.}
\end{equation*} 
Therefore, $\sig=\abs{\dot\ga}$, $\ta_\ga=\ta$ and $\no_\ga=\no$. It is thus  
natural to call $ \sig $ and $\ka$ the \tdfn{speed} and
\tdfn{curvature} of $\ga$, even though $ \sig,\,\ka\in L^2[0,1] $.

\begin{urmk} 
	Although $\dot \ga=\sig\ta$ is, in general, defined only almost
	everywhere on $[0,1]$, if we reparametrize $\ga$ by arc-length then it
	becomes a regular curve, because $\ga'=\ta$ is continuous. It is helpful to 
	regard  admissible curves simply as regular curves whose curvatures
	are defined a.e..
\end{urmk}

\begin{dfn}\label{D:loose} 
	For $u\in UTS$, let $\sr
	LS_{\ka_1}^{\ka_2}(u,\cdot)$ be the set of all $(\ka_1,\ka_2)$-admissible
	curves $\ga\colon [0,1]\to S$ with $\ta_\ga(0)=u$. 
\end{dfn}
If $ S $ is complete, then this set is identified with $\L$ via the
correspondence $\ga\leftrightarrow (\hat \sig,\hat \ka)$, thus furnishing $\sr
LS_{\ka_1}^{\ka_2}(u,\cdot)$ with a trivial Hilbert manifold structure. If $ S $
is not complete, then $ \sr LS_{\ka_1}^{\ka_2}(u,\cdot) $ is some mysterious
open subset of $ \L $.  However, we still have the following. 

\begin{lem}\label{L:mysterious}
	For all $ u \in UTS $, $ \sr LS_{\ka_1}^{\ka_2}(u,\cdot) $ is homeomorphic
	to $ \L $.
\end{lem}
\begin{proof}
	The proof is almost identical to that of \lref{L:contractible}\?(a). Given a
	family of curves indexed by a compact space, first reparametrize all curves
	by constant speed and shrink them to a common length $ \la $ smaller than the
	injectivity radius of $ S $ at the basepoint of $ u $. Now each curve is
	completely determined by its curvature, and conversely any $ L^2
	$-function $ \hat \ka \colon [0,1] \to \R $ determines a unique curve of
	constant speed $ \la $ having $ u $ for its initial unit tangent
	vector, via \eqref{E:Sobolev}. But the set of all such functions is
	convex. Thus $ \sr LS_{\ka_1}^{\ka_2}(u,\cdot) $ is weakly contractible,
	hence homeomorphic to $ \L $ by \lref{L:Hilbert}\?(a).
\end{proof}

\begin{lem}\label{L:submersion} 
	Let $ S $ be a surface and define $F\colon \sr
	LS_{\ka_1}^{\ka_2}(u,\cdot)\to UTS$ by $\ga\mapsto \ta_\ga(1)$. Then $F$ is
	a submersion, and consequently an open map.  
\end{lem}

\begin{proof}
	The proof when $ S = \R^2 $ is given in \cite{SalZueh1}, Lemma 1.5. The
	proof in the general case follows by considering Riemannian normal
	coordinates, which are flat at least to second order, in a neighborhood of
	the basepoint of $ \ta_\ga(1) $.
\end{proof}

\begin{dfn}[spaces of admissible curves]\label{D:Lurvespace} 
	Define $\sr LS_{\ka_1}^{\ka_2}(u,v)$ to be the subspace of $\sr
	LS_{\ka_1}^{\ka_2}(u,\cdot)$ consisting of all $ \ga $ such that
	$\ta_\ga(1)=v$.  
\end{dfn}

It follows from (\ref{L:submersion}) that $\sr LS_{\ka_1}^{\ka_2}(u,v)$ is a
closed submanifold of codimension 3 in $\sr LS_{\ka_1}^{\ka_2}(u,\cdot)\home \L$,
provided it is not empty. We have already seen in \cref{C:closed} that such
spaces may indeed be empty, but this cannot occur when $ S $ is compact;
cf.~\tref{T:compact} and \lref{L:dense}.

\subsection*{Relations between spaces of curves}

If $(\ka_1,\ka_2)\subs (\bar \ka_1,\bar \ka_2)$ and $\ga\in \sr
LS_{\ka_1}^{\ka_2}(u,v)$, then we can also consider $\ga$ as a curve in $\sr
LS_{\bar\ka_1}^{\bar\ka_2}(u,v)$. However, the
topology of the former space is strictly finer (i.e., has more open sets) than
the topology induced by the resulting inclusion. 
\begin{lem}\label{L:LtoLinclusion} 
	Let $(\ka_1,\ka_2)\subs (\bar \ka_1,\bar \ka_2)$,  $S$ be a surface and
	$u\in UTS$. Then 
	\begin{equation}\label{E:pair} 
		(\hat \sig,\hat \ka)\mapsto 
		\big(\hat \sig, h_{\bar \ka_1,\bar\ka_2}\circ 
		h_{\ka_1,\ka_2}^{-1}\circ \hat \ka\big)
	\end{equation} 
	defines a continuous injection $j\colon \sr LS_{\ka_1}^{\ka_2}(u,\cdot)\to
	\sr LS_{\bar\ka_1}^{\bar\ka_2}(u,\cdot)$. The actual curves on $S$
	corresponding to these pairs are the same, but $j$ is not a topological
	embedding unless $\bar \ka_1=\ka_1$ and $\bar \ka_2=\ka_2$.  
\end{lem}

\begin{proof} 
	It may be assumed that $ S $ is oriented. The curve $\ga$
	corresponding to $(\hat \sig,\hat \ka)\in \sr LS_{\ka_1}^{\ka_2}(u,v)$ is
	obtained as the solution of \eqref{E:de} with
	\begin{equation*}
		\sig=h^{-1}\circ \hat \sig \text{\quad and\quad
		}\ka=h^{-1}_{\ka_1,\,\ka_2}\circ \hat\ka.  
	\end{equation*} 
	The curve $\eta$ corresponding to the right side of \eqref{E:pair} in $\sr
	LS_{\bar\ka_1}^{\bar\ka_2}(u,\cdot)$ is the solution of \eqref{E:de} with
	\begin{equation*}
		\sig=h^{-1}\circ \hat \sig \text{\quad and\quad
		}\ka=h^{-1}_{\bar\ka_1,\,\bar\ka_2}\circ\big(h_{\bar
		\ka_1,\bar\ka_2}\circ h_{\ka_1,\ka_2}^{-1}\circ \hat
		\ka\big)=h^{-1}_{\ka_1,\,\ka_2}\circ \hat\ka.  
	\end{equation*} 
	By uniqueness of solutions, $\ga=\eta$. In particular, $j$ is injective. 

	Set $g=h_{\bar \ka_1,\bar\ka_2}\circ h_{\ka_1,\ka_2}^{-1}$. Observe that
	\begin{equation*}
		\lim_{t\to+ \infty}g'(t)=
		\begin{cases} 
			1 & \text{ if $\bar\ka_2=\ka_2$;} \\ 
			0 & \text{ otherwise;} 
		\end{cases}
		\qquad 
		\lim_{t\to- \infty}g'(t)=
		\begin{cases} 
			1 & \text{ if $\bar\ka_1=\ka_1$;} \\ 
			0 & \text{ otherwise.} 
		\end{cases} 
	\end{equation*} 
	Hence, $\abs{g'}$ is bounded over $\R$. Consequently, there exists $C>0$
	such that
	\begin{equation*}
		\norm{g\circ f_1-g\circ f_2}_2\leq C\norm{f_1-f_2}_2\ \ 
		\text{for any $f_1,\,f_2\in L^2[0,1]$.}
	\end{equation*} 
	We conclude that $j\colon (\hat\sig,\hat\ka)\mapsto
	(\hat\sig,g\circ \hat\ka)$ is continuous.

	Suppose now that $(\ka_1,\ka_2)\psubs (\bar\ka_1,\bar\ka_2)$. No generality
	is lost in assuming that $\ka_2<\bar\ka_2$. Let 
	\begin{equation*}
		m=g(0) \text{\ \ and \ \ } M=g(+\infty)=h_{\bar\ka_1,\bar\ka_2}(\ka_2).
	\end{equation*} 
	Define a sequence of $L^2$ functions $\hat \ka_n\colon [0,1]\to \R$ by: 
	\begin{equation*}
		\quad \hat\ka_n(t)=
		\begin{cases} 
			n & \text{ if $t\in \big[\frac{1}{2}-\frac{1}{2n}\,
			,\,\frac{1}{2}+\frac{1}{2n}\big]$;} \\ 
			0 & \text{ otherwise. } 
		\end{cases}\qquad (n\in \N^+,~t\in [0,1]).  
	\end{equation*} 
	Since $g$ is the composite of increasing functions, $g(t)<g(+\infty)=M$ for
	any $t\in \R$. Therefore,
	\begin{equation*}
		\abs{g\circ \hat\ka_n(t)-m}
		\begin{cases} 
			\leq M-m & \text{ if\, $t\in \big[\frac{1}{2}-\frac{1}{2n}\,,\,
			\frac{1}{2}+\frac{1}{2n}\big]$;} \\ 
			= 0 & \text{ otherwise}.  
		\end{cases} 
	\end{equation*}
	Hence, $\norm{\hat {\ka}_n}_2=\sqrt{n}\to +\infty$ as $n$ increases, while
	$\norm{g\circ \hat \ka_n-m}_2\leq n^{-\frac{1}{2}}(M-m) \to 0$.  We conclude
	that $ j $ is not a topological embedding.  This
	argument may be modified to prove that  $\sr
	LS_{\ka_1}^{\ka_2}(u,v) \inc \sr LS_{\bar\ka_1}^{\bar\ka_2}(u,v)$ is likewise
	not an embedding for any $v$.
\end{proof}

\begin{lem}\label{L:CtoLinclusion} 
	Let $(\ka_1,\ka_2)\subs (\bar \ka_1,\bar \ka_2)$ and $u,\,v\in UTS$. Then
	\begin{equation}\label{E:injection} 
		j\colon \sr CS_{\kappa_1}^{\kappa_2}(u,v)^r\to \sr
		LS_{\bar\kappa_1}^{\bar\kappa_2}(u,v),\quad \ga\mapsto (h\circ \abs{\dot
		\ga} \,,\, h_{\bar\ka_1}^{\bar\ka_2}\circ \ka_\ga)
	\end{equation} 
	is a continuous injection, but not an embedding, for all $r\geq 2$.
	Moreover, the actual curve on $S$ corresponding to $ j(\ga) $ is $\ga$
	itself.  
\end{lem}
\begin{proof} 
	The proof is very similar to that of (\ref{L:LtoLinclusion}). 
\end{proof}

The following lemmas contain all the results on infinite-dimensional manifolds
that we shall need.\footnote{\lref{L:Hilbert} and a weaker version of
\lref{L:net} have already appeared in \cite{SalZueh}.} 

\begin{lem}\label{L:Hilbert}
	Let $\sr M,\,\sr N$ be (infinite-dimensional) Banach manifolds. Then:	
	\begin{enumerate} 
		\item [(a)] If $\sr M,\,\sr N$ are weakly homotopy equivalent, then they
			are in fact homeomorphic (diffeomorphic if $\sr M,\,\sr N$ are
			Hilbert manifolds). 
		\item [(b)] If the Banach manifold $ \sr M $ and the finite-dimensional
			manifold $ M $ are weakly homotopy equivalent, then $ \sr M $ is
			homeomorphic to $ M \times \L $; in particular, $ \sr M $ and $ M $
			are homotopy equivalent.
		\item [(c)] Let $E$ and $F$ be separable Banach
			spaces.  Suppose $i\colon F\to E$ is a bounded,
			injective linear map with dense image and $\sr M\subs E$ is
			a smooth closed submanifold of finite codimension.  Then $\sr
			N=i^{-1}(\sr M)$ is a smooth closed submanifold of\, $F$
			and $i\colon \sr N\to \sr M$ is a homotopy equivalence.  
	\end{enumerate} 
\end{lem}

\begin{proof} 
	Part (a) follows from thm.~15 in \cite{Palais} and cor.~3
	in \cite{Henderson}. For part (b), apply (a) to $ \sr M $ and $
	\sr N = M \times \L $. Part (c) is thm.~2 in
	\cite{BurSalTom}. 
\end{proof}

\begin{lem}\label{L:net} 
	Let $\L$ be a separable Hilbert space, $D\subs \L$ a
	dense vector subspace,  $L\subs \L$ a submanifold of finite codimension and
	$U$ an open subset of $L$.  Then the set inclusion $D\cap U\to U$ is a weak
	homotopy equivalence.  
\end{lem} 
\begin{proof} 
	We shall prove the lemma
	when $L=h^{-1}(0)$ for some submersion $h\colon V\to \R^n$, where $V$ is an
	open subset of $\L$. This is sufficient for our purposes and the general
	assertion can be deduced from this by using a partition of unity subordinate
	to a suitable cover of $L$.

	Let $T$ be a tubular neighborhood of $U$ in $V$ such that $T\cap L=U$. Let
	$K$ be a compact simplicial complex and $f\colon K\to U$ a continuous map.
	We shall obtain a continuous $H\colon [0,2]\times K\to U$ such that
	$H(0,a)=f(a)$ for every $a\in K$ and $H(\se{2}\times K)\subs D\cap U$. Let
	$e_j$ denote the $j$-th vector in the canonical basis for $\R^n$,
	$e_0=-\sum_{j=1}^ne_j$ and let $\De\subs \R^n$ denote the $n$-simplex
	$[e_0,\dots,e_n]$. Let  $[x_0,x_1,\dots,x_n]\subs T$ be an $n$-simplex and
	$\vphi\colon \De\to [x_0,x_1,\dots,x_n]$ be given by
	\begin{equation*}
		\vphi\bigg(\sum_{j=0}^n s_je_j\bigg)=\sum_{j=0}^n s_j x_j, \text{\ \
			where\ \ $\sum_{j=0}^n s_j=1$\ \ and\ \ $s_j\geq 0$\ \ for all\ \
		$j=0,\dots,n$}.  
	\end{equation*}
	We shall say that $[x_0,x_1,\dots,x_n]$ is \tdfn{neat} if $h\circ
	\vphi\colon \De\to \R^n$ is an embedding and $0\in (h\circ \vphi)(\Int
	\De)$. 

	Given $p\in T$, let $dh_p$ denote the derivative of $h$ at $p$ and
	$N_p=\ker(dh_p)$. Define $w_j\colon T\to \L$ by:
	\begin{equation}\label{E:w_j}
		w_j(p)=\big(dh_p|_{N_p^\perp}\big)^{-1}(e_j)\quad (p\in T,~j=0,\dots,n).
	\end{equation} 
	Notice that 
	$h\big(p+\sum_j\la_j w_j(p)\big)=h(p)+ \sum_j\la_j e_j+o(\abs{\la})$ 
	(for $\la=(\la_0,\dots,\la_n)$ and $p\in T$).
	Hence, using compactness of $K$, we can find $r,\eps>0$ such that:
	\begin{enumerate} 
		\item [(i)] For any $p\in f(K)$, $[p+rw_0(p),\dots,p+rw_n(p)]\subs T$
		and it is neat; 
		\item [(ii)] If $p\in f(K)$ and $\abs{q_j-(p+rw_j(p))}<\eps$ for each
			$j$, then $[q_0,\dots,q_n]\subs T$ and it is neat.  
	\end{enumerate} 
	Let $a_i$ ($i=1,\dots,m$) be the vertices of the triangulation of $K$. Set
	$v_i=f(a_i)$ and 
	\begin{equation*}
		v_{ij}=v_i+rw_j(v_i)\quad
		(i=1,\dots,m,~j=0,\dots,n).  
	\end{equation*} 
	For each such $i,j$, choose $\te{v}_{ij}\in D\cap T$ with
	$\abs{\te{v}_{ij}-v_{ij}}<\frac{\eps}{2}$. Let
	\begin{alignat}{9} 
		&v_{ij}(s)=(2-s)v_{ij}+(s-1)\te{v}_{ij},\ 
		\text{ so that } \notag \\ \label{E:simest} 
		&\abs{v_{ij}(s)-v_{ij}}<\frac{\eps}{2}\ \ 
		(s\in [1,2],~i=1,\dots,m,~j=0,\dots,n).  
	\end{alignat} 
	For any $i,i'\in
	\se{1,\dots,m}$ and $j=0,\dots,n$, we have 
	\begin{equation*}
		\abs{v_{ij}-v_{i'j}}\leq 
		\abs{f(a_{i})-f(a_{i'})}+r\abs{w_j\circ f(a_{i})-w_j\circ f(a_{i'})}.
	\end{equation*} 
	Since $f$ and the $w_j$ are continuous functions, we can suppose that the
	triangulation of $K$ is so fine that $\abs{v_{ij}-v_{i'j}}<\frac{\eps}{2}$
	for each $j=0,\dots,n$ whenever there exists a simplex having $a_{i}$,
	$a_{i'}$ as two of its vertices.  Let $a\in K$ lie in some $d$-simplex of
	this triangulation, say, $a=\sum_{i=1}^{d+1}t_ia_i$ (where each $t_i>0$ and
	$\sum_i t_i=1$). Set
	\begin{equation*}
		z_{j}(s)=\sum_{i=1}^{d+1}t_iv_{ij}(s)\quad (s\in [1,2],~j=0,\dots,n).  
	\end{equation*} 
	Then $[z_0(s),\dots,z_n(s)]$ is a neat simplex because condition (ii) is
	satisfied (with $p=v_1$): 
	\begin{equation*}
		\bigg\vert \sum_{i=1}^{d+1}t_iv_{ij}(s)-v_{1j}\bigg\vert\leq
		\sum_{i=1}^{d+1}t_i \big( \abs{v_{ij}(s)-v_{ij}} + \abs{v_{ij}-v_{1j}}
		\big) < \eps,
	\end{equation*} 
	the strict inequality coming from \eqref{E:simest} and our hypothesis on the
	triangulation.  Define $H(s,a)$ as the unique element of $h^{-1}(0)\cap
	[z_0(s),\dots,z_n(s)]$ ($s\in [1,2]$). Observe that for any $a\in K$,
	$H(s,a)\in U=h^{-1}(0)\cap T$ $(s\in [1,2]$) and $H(2,a)\in D\cap U$, as it
	is  the convex combination of the $\te{v}_{ij}\in D$.

	By reducing $r,\eps>0$ (and refining the triangulation of $K$) if necessary,
	we can ensure that 
	\begin{equation*}
		(1-s)f(a)+sH(1,a)\in T\quad
		\text{for all $s\in [0,1]$ and $a\in K$.} 
	\end{equation*} 
	Let $\pr\colon T\to U$ be the associated retraction. Complete the definition
	of $H$ by setting: 
	\begin{equation*}
		H(s,a)=\pr \big((1-s)f(a)+sH(1,a)\big)\quad 
		\text{($s\in [0,1]$,~$a\in K$).}
	\end{equation*} 
	The existence of $H$ shows that $f$ is homotopic within $U$
	to a map whose image is contained in $D\cap U$. Taking $K=\Ss^k$, we
	conclude that the set inclusion $D\cap U\to U$ induces surjective maps
	$\pi_k(D\cap U)\to \pi_k(U)$ for all $k\in \N$.

	We now establish that the inclusion $D\cap U\to U$ induces injections on all
	homotopy groups. Let $k\in \N$, $G\colon \D^{k+1}\to U$ be continuous and
	suppose that the image of $g=G|_{\Ss^k}$ is contained in $D\cap U$. Let
	$G_0\colon \D^{k+1}\to D\cap U$ be a close approximation to $G$; the
	existence of $G_0$ was proved above. Let $\eps\in (0,1)$ and define
	\begin{equation*}
		G_1\colon \D^{k+1}\to D\cap T\text{\ \ by\ \ }G_1(a)=
		\begin{cases}
			(1-s)g\big(\tfrac{a}{\abs{a}}\big)+sG_0\big(\tfrac{a}{\abs{a}}\big)
			& \text{\ \  if\ \ $\abs{a}=(1-s\eps)$,~$s\in [0,1]$} \\
			G_0\big(\tfrac{a}{1-\eps}\big) & \text{\ \  if\ \ $\abs{a}\leq
		1-\eps$ } 
		\end{cases}
	\end{equation*} 
	Notice that we can make $G_1$ as close as desired to $G$ by a suitable
	choice of $G_0$ and $\eps$. Let $w_j$ be as in \eqref{E:w_j}. We claim that
	there exist continuous functions $\te{w}_j\colon \D^{k+1}\to D$
	$(j=0,\dots,n)$ such that: 
	\begin{enumerate}
		\item [(iii)] $\sum_{j=0}^n \te{w}_j(a)=0$ for all $a\in \D^{k+1}$; 
		\item [(iv)] For any $a\in \D^{k+1}$,
			$[G_1(a)+\te{w}_0(a),\dots,G_1(a)+\te{w}_n(a)]\subs D\cap T$ and it
			is neat.  
	\end{enumerate} 
	To prove this, invoke condition (ii) above (with $\D^{k+1}$ in place of $K$
	and $G$ in place of $f$) together with denseness of $D$ to find constant
	$\te{w}_j$ on open sets which cover $\D^{k+1}$, and use a partition of
	unity. By (iv), for each $a\in \D^{k+1}$ there exist unique
	$t_0(a),\dots,t_n(a)\in [0,1]$ such that $\sum_{i}t_i(a)=1$ and
	\begin{equation*}
		G_2(a)=G_1(a)+t_0(a)\te{w}_0(a)+\dots+t_n(a)\te{w}_n(a)\in h^{-1}(0).
	\end{equation*} 
	We obtain thus a continuous map $G_2\colon \D^{k+1}\to D\cap U$.  Since
	${G_1}|_{\Ss^k}=g$ and $h\circ g=0$, we conclude from (iii) and uniqueness of
	the $t_i$ that $G_2|_{\Ss^k}=g$. Therefore, $G_2$ is a nullhomotopy of $g$
	in $D\cap U$.  
\end{proof}

\begin{crl}\label{L:dense} 
	The subset of all smooth curves in $\sr LS_{\ka_1}^{\ka_2}(u,v)$ is dense in
	the latter.  
\end{crl} 
\begin{proof}
	Take $\L=L^2[0,1]\times L^2[0,1]$, $D=C^\infty[0,1]\times C^\infty[0,1]$ and
	$U$ an open subset of $L=\sr LS_{\kappa_1}^{\kappa_2}(u,v)$. Then it is a
	trivial consequence of (\ref{L:net}) that $D\cap U\neq \emptyset$ if $U\neq
	\emptyset$.
\end{proof}

\begin{crl}[smooth approximation]\label{L:smoothie}
	Let $\sr U\subs \sr LS_{\ka_1}^{\ka_2}(u,v)$ be open, $K$ be a compact
	simplicial complex and $f\colon K\to \sr U$ a continuous map. Then there
	exists a continuous $g\colon K\to \sr U$ such that: 
\begin{enumerate} 
	\item [(i)] $f\iso g$ within $\sr U$.  
	\item [(ii)] $g(a)$ is a smooth curve for all $a\in K$.  
	\item [(iii)] All derivatives of $g(a)$ with respect to $t$ depend
		continuously on $a\in K$.
\end{enumerate} 
	Thus, the map $j\colon \sr CS_{\ka_1}^{\ka_2}(u,v)\to \sr
	LS_{\ka_1}^{\ka_2}(u,v)$ in \eqref{E:injection} induces surjections
	$\pi_k(j^{-1}(\sr U))\to \pi_k(\sr U)$ for all $k\in \N$.  
\end{crl} 
\begin{proof} 
	Parts (i) and (ii) are exactly what was established in the first part of the
	proof of (\ref{L:net}), in the special case where $\L=L^2[0,1]\times
	L^2[0,1]$, $D=C^\infty[0,1]\times C^\infty[0,1]$,  $L=\sr
	LS_{\kappa_1}^{\kappa_2}(u,v)$ and $U=\sr U$. The image of the function
	$g=H_2\colon K\to \sr U$ constructed there is contained in a
	finite-dimensional vector subspace of $D$, viz., the one generated by all
	$\te{v}_{ij}$, so (iii) also holds.  
\end{proof}

\begin{lem}[$ \sr C\home \sr L $]\label{L:C^2} 
	Let $ S $ be complete. Then the inclusion $i\colon \sr
	CS_{\ka_1}^{\ka_2}(u,v)^r\to \sr LS_{\ka_1}^{\ka_2}(u,v)$ is a homotopy
	equivalence for any $r\geq 2$.  Consequently, $\sr
	CS_{\kappa_1}^{\kappa_2}(u,v)^r$ is homeomorphic to $\sr
	LS_{\kappa_1}^{\kappa_2}(u,v)$. 
\end{lem}
\begin{proof}
	Let $\L=L^2[0,1]\times L^2[0,1]$, let $F=C^{r-1}[0,1]\times
	C^{r-2}[0,1]$ (where $C^k[0,1]$ denotes the set of all $C^k$ functions
	$[0,1]\to \R$, with the $C^k$ norm) and let $i\colon F\to \L$ be set
	inclusion. Setting $\sr M=\sr LS_{\ka_1}^{\ka_2}(u,v)$, we conclude from
	(\ref{L:Hilbert}\?(c)) that $i\colon \sr N=i^{-1}(\sr M)\inc \sr M$ is a
	homotopy equivalence. We claim that $\sr N$ is homeomorphic to $\sr
	CS_{\ka_1}^{\ka_2}(u,v)^r$, where the homeomorphism is obtained by
	associating a pair $(\hat \sig,\hat \ka)\in \sr N$ to the curve $\ga$
	obtained by solving \eqref{E:de}, with $\sig$ and $\ka$ as in
	\eqref{E:Sobolev}.

	Suppose first that $\ga\in \sr CS_{\ka_1}^{\ka_2}(u,v)^r$. Then
	$\abs{\dot\ga}$ (resp.~$\ka$) is a function $[0,1]\to \R$ of class $C^{r-1}$
	(resp.~$C^{r-2}$). Hence, so are $\hat{\sig}=h\circ \abs{\dot\ga}$ and $\hat
	\ka=h_{\ka_1}^{\ka_2}\circ \ka$, since $h$ and $h_{\ka_1}^{\ka_2}$ are
	smooth. Moreover, if $\ga,\,\eta\in \sr CS_{\ka_1}^{\ka_2}(u,v)^r$ are close
	in $C^r$ topology, then $\hat \ka_\ga$ is $C^{r-2}$-close to $\hat\ka_\eta$
	and $\hat \sig_\ga$ is $C^{r-1}$-close to $\hat \sig_\eta$.

	Conversely, if $(\hat{\sig},\hat{\ka})\in \sr N$, then $\sig=h^{-1}\circ
	\hat \sig$ is of class $C^{r-1}$ and $\ka=(h_{\ka_1}^{\ka_2})^{-1}\circ \hat
	\ka$ of class $C^{r-2}$. Since all functions on the right side of
	\eqref{E:de} are of class (at least) $C^{r-2}$, the solution $\ta=\ta_\ga$
	to this initial value problem is of class $C^{r-1}$. Moreover,
	$\dot\ga=\sig\ta$, hence the velocity vector of $\ga$ is seen to be of class
	$C^{r-1}$. We conclude that $\ga$ is a curve of class $C^r$. Further,
	continuous dependence on the parameters of a differential equation shows
	that the correspondence $(\hat \sig,\hat \ka)\mapsto \ta_\ga$ is continuous.
	Since $\ga$ is obtained by integrating $\sig\ta_\ga$, we deduce that  the
	map $(\hat\sig,\hat\ka)\mapsto \ga$ is likewise continuous.

	The last assertion of the lemma follows from (\ref{L:Hilbert}\?(c)).
\end{proof}


\section*{Acknowledgements} 

The first author is partially supported by grants from \ltsc{capes}, \ltsc{cnpq}
and \ltsc{faperj}.  
The second author gratefully acknowledges C.~Gorodski, \ltsc{ime-usp} and
\ltsc{unb} for hosting him as a post-doctoral fellow, and \ltsc{fapesp}
(grant 14/22556-3) and \ltsc{capes} for the financial support.
The authors thank the anonymous referee for her/his helpful comments. 

\vfill\eject

\bibliography{references.bib}
\bibliographystyle{amsplain}

\vspace{12pt}
\noindent
{\href{mailto:nicolau@mat.puc-rio.br}{\ttt{nicolau@mat.puc-rio.br}} \\
\small \tsc{Departamento de Matem\'atica, Pontif\'icia
Universidade Cat\'olica do Rio de Janeiro (\ltsc{puc-rio})\\ 
Rua Marqu\^es de S\~ao Vicente 225, G\'avea -- 
Rio de Janeiro, RJ 22453-900, Brazil}}

\vspace{12pt} 

\noindent 
{\href{mailto:pedroz@ime.usp.br}{\ttt{pedroz@ime.usp.br}} \\
\noindent{\small \tsc{Instituto de Matem\'atica e Estat\'istica,
	Universidade de S\~ao Paulo (\ltsc{ime-usp}) \\ Rua~do Mat\~ao 1010, Cidade
	Universit\'aria -- S\~ao Paulo, SP
05508-090, Brazil}} 

\vspace{4pt} 

\noindent \small \tsc{Departamento de Matem\'atica, Universidade de Bras\'ilia
(\ltsc{unb}) \\ Campus Darcy Ribeiro, 70910-900 -- Bras\'ilia, DF, Brazil}

\end{document}